\documentclass{amsart}

\usepackage[utf8]{inputenc}
\usepackage{amsthm}
\usepackage{amssymb}
\usepackage{amsmath}
\usepackage{stmaryrd}
\usepackage{cite}
\usepackage{hyperref}
\usepackage{MnSymbol} 
\usepackage[all]{xy} 
\usepackage{leftidx}

\theoremstyle{plain}
\newtheorem{mythm}{Theorem}
\newtheorem*{thm1}{Theorem 1}
\newtheorem*{thm pos}{Theorem 1.1 (Existence of dual Specht filtrations)}
\newtheorem*{thm neg}{Conjecture (Non-existence of dual Specht filtrations)}
\newtheorem*{thm1*}{Theorem 1, Version 2}
\newtheorem{prop}{Proposition}
\newtheorem*{thm}{Theorem}
\newtheorem{lem}[prop]{Lemma}
\newtheorem{cor}[prop]{Corollary}
\newtheorem*{cor*}{Corollary \ref{cor: thm 1}}

\theoremstyle{definition}

\newtheorem*{ex}{Example}
\newtheorem{dfprop}[prop]{Definition/Proposition}

\theoremstyle{remark}
\newtheorem*{rk}{Remark}

\newcommand{\arcd}{\ar@{-}@/_/} 
\newcommand{\arcu}{\ar@{-}@/^/} 
\newcommand{\arcU}{\ar@{-}@/^10pt/} 
\newcommand{\tra}{\ar@{-}} 

\newcommand{\xysmall}{\xymatrixrowsep{5pt}\xymatrixcolsep{10pt}\xymatrix} 

\renewcommand{\top}{\text{top}}
\newcommand{\bottom}{\text{bottom}}

\newcommand{\tensorover}[1]{\underset{\scriptscriptstyle #1}{\otimes}}

\renewcommand{\mod}{\text{\textnormal{mod}}}

\renewcommand{\th}{\text{\textnormal{\small th}}\hspace*{.6ex}}

\begin{document}

\title{Restricting cell modules of partition algebras}
\author{Inga Paul}
\address{Institut f\"ur Algebra und Zahlentheorie, Universit\"at Stuttgart\newline \phantom{PF} Pfaffenwaldring 57, 70569 Stuttgart, Germany}
\email{inga.paul@mathematik.uni-stuttgart.de}

\begin{abstract}
The restriction of a (dual) Specht module to a smaller symmetric group has a filtration by (dual) Specht modules of this smaller group.  In the cellular structure of the group algebra of the symmetric group, the cell modules are exactly the (dual) Specht modules. The partition algebra is a cellular algebra containing the group algebra of the symmetric group. In this article, we study the structure of the restriction of a cell module to the group algebra of a symmetric group (with smaller index) and find conditions for the restriction to possess a (dual) Specht filtration. 
\end{abstract}

\maketitle

\section{Introduction}\label{sec: introduction}
The partition algebra was independently defined by Martin \cite{PMartin} and Jones \cite{Jones} to describe the Potts model in statistical mechanics.
In representation theory, the partition algebra $P_k(r,\delta)$ arises as a diagram algebra containing the Temperley-Lieb and Brauer algebras. It has nice structural properties such as being cellular \cite{Xi}   
and quasi-hereditary if and only if $\delta \neq 0$ and $\text{\textnormal{char}} k = 0$ or $\text{\textnormal{char}} k > r$ \cite{KXqh}. For $r \geq n$, there is Schur-Weyl duality between $k\Sigma_n$ and $P_k(r,n)$, see for example \cite{HalversonRam}. 

An enhanced cellular structure, called \emph{cellularly stratified}, ensures that the cell modules of $P_k(r,\delta)$, with $\delta \neq 0$, arise from cell modules of the group algebras of symmetric groups with index $n \leq r$ by induction, i.e. they are induced (dual) Specht modules \cite{HHKP}. The same holds for the cell modules of the Brauer algebra. When restricted to a group algebra of a symmetric group,  the cell modules of a Brauer algebra admit a filtration by (dual) Specht modules, as shown in \cite{P}. The approach used for Brauer algebras is not applicable for the partition algebra, since it is based on the fact that, in a Brauer diagram, each dot is connected to exactly one other dot. For the partition algebra, there is no such regularity, which makes the problem more complex. 

In this article, we consider the question when the restriction $e_lAe_n \tensorover{e_nAe_n} S_\nu$ of a cell module $Ae_n \tensorover{e_nAe_n} S_\nu$ of the partition algebra $A=P_k(r,\delta)$ to a group algebra of a symmetric group $\Sigma_l$ with $0\leq l \leq r$ admits a dual Specht filtration. The main result is the following.

\begin{mythm}\label{main theorem}
Let $A$ be the partition algebra $P_k(r,\delta)$ with $\delta \neq 0$. Let $0\leq n \leq l \leq r$. Let $\nu$ be a partition of $n$. 
The restriction $e_lAe_n \tensorover{e_nAe_n}S_\nu$ of the cell module $Ae_n\tensorover{e_nAe_n}S_\nu$ to $k\Sigma_l -\mod$ admits a dual Specht filtration if the following holds
\begin{enumerate}
\item\label{main thm assumption char} $\text{\textnormal{char}} k = 0$ or $\text{\textnormal{char}} k > \lfloor \frac{l-n}{3} \rfloor$ \end{enumerate}
 and
\begin{enumerate}\setcounter{enumi}{1}
\item\label{main thm assumption gFm} for all $i \in \{1,...,l\}, \alpha_i \in \{1,...,\lfloor \frac{l}{i}\rfloor\}$ and for all partitions $\lambda^i$ of $\alpha_i$, the generalised Foulkes modules with inner twists $$k\Sigma_{i\alpha_i} \tensorover{k(\Sigma_i \wr \Sigma_{\alpha_i})} S_{\lambda^i} \simeq (k\Sigma_{i\alpha_i} \tensorover{k\Sigma_{(i^{\alpha_i})}} k) \tensorover{k\Sigma_{\alpha_i}} S_{\lambda^i}$$ admit filtrations by dual $k\Sigma_{i\alpha_i}$-Specht modules.
\end{enumerate}
\end{mythm}
The elements $e_n$ and $e_l$ are special idempotents needed for the cellularly stratified structure of $A$. They are defined in Subsection \ref{subsec: PA structure}.  Generalised Foulkes modules with inner twists are defined in Section \ref{sec: result} and the bimodule structure of the permutation module $k\Sigma_{i\alpha_i} \tensorover{k\Sigma_{(i^{\alpha_i})}} k$ is defined in Subsection \ref{subsec: permutation bimodule}.

Section \ref{sec: definition} contains a definition of the partition algebra, some notation and an explanation of the cellular and cellularly stratified structures. 

The technical main result, Theorem \ref{TMR}, is the structural study of the $(k\Sigma_l,k\Sigma_n)$-bimodule $e_l(A/Ae_{n-1}A)e_n$, which is the content of Section \ref{sec: restriction cell modules}. It is the main ingredient for the proof of Theorem \ref{main theorem} in Section \ref{sec: result}. In Subsection \ref{subsec: notation}, we define some notation and prove the first result - a decomposition of $e_l(A/Ae_{n-1}A)e_n$ into direct summands, where each summand is indexed by an equivalence class of partial diagrams (with exactly $n$\, \emph{labelled} parts). We fix such a partial diagram and continue with the study of the corresponding summand. We realise this summand as an induced exterior tensor product of modules corresponding to the labelled and the unlabelled dots respectively, which we study individually in Subsections \ref{subsec: Foulkes module} and \ref{subsec: permutation bimodule}. We show that the module corresponding to the unlabelled dots is isomorphic to an induced exterior tensor product of \emph{Foulkes modules} $k\Sigma_{am} \tensorover{k(\Sigma_a \wr \Sigma_m)} k$. If $\text{\textnormal{char}}k > m$ (or $\text{\textnormal{char}} k = 0$), we can show that the Foulkes module $k\Sigma_{am} \tensorover{k(\Sigma_a \wr \Sigma_m)} k$ has a dual Specht filtration; in smaller characteristic this is an open problem. This leads to assumption (\ref{main thm assumption char}) in Theorem \ref{main theorem}. In Subsection \ref{subsec: permutation bimodule}, we show that the module corresponding to the labelled dots is isomorphic to an induced exterior tensor product of Young permutation modules $M^{(i^{\alpha_i})} = k\Sigma_{i\alpha_i} \tensorover{k\Sigma_{(i^{\alpha_i})}} k$ with a $(k\Sigma_{i\alpha_i},k\Sigma_{\alpha_i})$-bimodule structure.
In Subsection \ref{subsec: eAe Specht}, we show that $e_l(A/Ae_{n-1}A)e_n$ admits a dual Specht filtration as left $k\Sigma_l$-module, provided $\text{\textnormal{char}}k$ is large enough.\\

The proof of Theorem \ref{main theorem} can be found in Section \ref{sec: result}. It uses the results from the previous section, the result on Brauer algebras from \cite{P} and the characteristic-free version of the Littlewood-Richardson rule \cite{JP}. We prove the theorem in Subsection \ref{subsec: proof positive} and state another version of it with nicer, yet stronger, assumptions as a corollary. 

\begin{cor*}
Let $A$ be the partition algebra $P_k(r,\delta)$, $\delta \neq 0$ and let $0\leq n \leq l \leq r$. Let $\nu$ be a partition of $n$. 
The restriction $e_lAe_n \tensorover{e_nAe_n}S_\nu$ of the cell module $Ae_n\tensorover{e_nAe_n}S_\nu$ to $k\Sigma_l -\mod$ admits a dual Specht filtration if $\text{\textnormal{char}} k = 0$ or $\text{\textnormal{char}} k > \lfloor \frac{l}{2} \rfloor$.
\end{cor*}

In general, dual Specht filtrations may not exist. Some evidence and an idea how to construct a counterexample are given in Subsection \ref{subsec: proof negative}. Subsection \ref{subsec: necessity of conditions} shows that some of the assumptions made in Theorem \ref{main theorem} are necessary. \\

This article is the first of two articles arising from the author's PhD thesis \cite{diss}. The aim of the thesis was to extend the construction of $k\Sigma_r$-permutation modules to permutation modules for a large class of diagram algebras, as it was done by Hartmann and Paget for Brauer algebras \cite{HP}, such that these `new' permutation modules satisfy similar properties to those of $k\Sigma_r$-permutation modules. A list of assumptions which the algebra $A$ has to satisfy was given; amongst others, it is assumed that the cell modules of the cellularly stratified algebra $A$ admit a dual Specht filtration when restricted to the module category of one of the input algebras $B_l$. In the case $A=P_k(r,\delta)$, these input algebras are exactly the group algebras of the symmetric groups $\Sigma_l$ with $0 \leq l \leq r$. The remaining assumptions are relatively easy to show for $A=P_k(r,\delta)$, which makes the article at hand the crucial ingredient for the definition of permutation modules for partition algebras. However, if the restriction of a cell module fails to admit a dual Specht filtration, we can still define permutation modules. They might not be relative projective with respect to the category of cell filtered modules anymore, as we show in \cite{permutationmodules}. This is the second article (in progress) arising from the author's PhD thesis, which is devoted to the general construction of permutation modules for cellularly stratified algebras, including a section about the partition algebra. Together with the article at hand, this makes the construction of permutation modules for partition algebras possible.

\section[Definition and Structure]{Definition of the Partition Algebra and its Structure}\label{sec: definition} 
\subsection{Definition}
Let $k$ be an algebraically closed field of arbitrary characteristic and let $\delta \in k$. Let $r \in \mathbb{N}$.

The \textit{partition algebra} $P_k(r,\delta)$ is the associative $k$-algebra with basis consisting of set partitions of $\{1,...,r,1',...,r'\}$. A \textit{set partition} of a set $X$ is a collection of pairwise disjoint subsets $X_i \subseteq X$, such that $\coprod X_i = X$.  Regarding $P_k(r,\delta)$ as a diagram algebra, this means that the basis consists of diagrams with two rows of $r$ dots each (top row labelled by $1,...,r$ and bottom row labelled by $1',...,r'$), where dots which belong to the same part of the partition are connected transitively. Note that this description is not unique.  
For example, the set partition $$\{\{1,2',3'\},\{2\},\{3,4,5,5',6'\},\{6,4'\},\{1'\} \}$$
corresponds, among others, to the diagram $$\xysmall{\bullet^1 \tra[dr] & \bullet^2 & \bullet^3 \tra[r] & \bullet^4 \tra[r] & \bullet^5 \tra[d] & \bullet^6 \tra[dll] \\
\bullet_{1'} & \bullet_{2'} \tra[r] & \bullet_{3'} & \bullet_{4'} & \bullet_{5'} \tra[r] & \bullet_{6'}}$$ 
as well as to the diagram $$\begin{minipage}{4cm}
\xysmall{\bullet^1 \tra[dr]\tra[drr] & \bullet^2 & \bullet^3 \tra[r] \tra[drr] & \bullet^4 \tra[r] & \bullet^5 \tra[dr] & \bullet^6 \tra[dll] \\
\bullet_{1'} & \bullet_{2'}  & \bullet_{3'} & \bullet_{4'} & \bullet_{5'} & \bullet_{6'}}
\end{minipage}$$
Multiplication is given by concatenation of diagrams, i.e. writing one diagram on top of the other, identifying the bottom row of the upper with the top row of the lower diagram and following the lines from top to bottom or within one row. Parts which have no dot in top or bottom row are replaced by a factor $\delta$. This multiplication is independent of the choice of diagram. 
We usually omit the labels $1,..,r,1',...,r'$. 

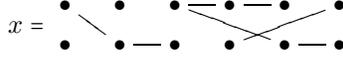
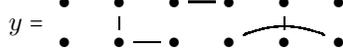
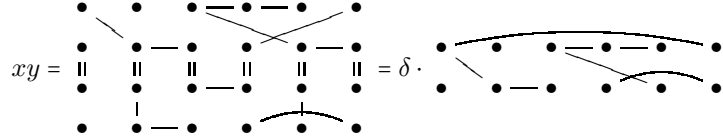
\begin{ex}\label{ex: PA multiplication}
 Let $$ x=
\begin{minipage}[c]{4cm} 
\xysmall{\bullet \tra[dr] & \bullet & \bullet \tra[r]\tra[drr] & \bullet \tra[r] & \bullet & \bullet \tra[dll] \\
\bullet & \bullet \tra[r] & \bullet & \bullet & \bullet \tra[r] & \bullet}
\end{minipage}$$ and $$
y=
\begin{minipage}[c]{4cm}
\xysmall{\bullet  & \bullet \tra[d] & \bullet \tra[r] & \bullet  & \bullet \tra[d] & \bullet \\
\bullet & \bullet \tra[r] & \bullet & \bullet \arcu[rr] & \bullet  & \bullet}\end{minipage}$$
Then 
$$xy= 
\begin{minipage}[c]{4.1cm}
\xysmall{\bullet \tra[dr] & \bullet & \bullet \tra[r]\tra[drr] & \bullet \tra[r] & \bullet & \bullet \tra[dll] \\
\bullet \ar@{=}[d] & \bullet \tra[r]\ar@{=}[d] & \bullet\ar@{=}[d] & \bullet\ar@{=}[d] & \bullet\ar@{=}[d] \tra[r] & \bullet\ar@{=}[d] \\
\bullet  & \bullet \tra[d] & \bullet \tra[r] & \bullet  & \bullet \tra[d] & \bullet \\
\bullet & \bullet \tra[r] & \bullet & \bullet \arcu[rr] & \bullet  & \bullet}
\end{minipage}
= \delta \cdot
\begin{minipage}[c]{4cm}
\xysmall{\bullet \arcu[rrrrr]\tra[dr] & \bullet & \bullet \tra[r]\tra[drr] & \bullet \tra[r]  & \bullet  & \bullet \\
\bullet & \bullet \tra[r] & \bullet & \bullet \arcu[rr] & \bullet  & \bullet}
\end{minipage}$$
\end{ex}

 We choose to write all diagrams as follows. First, connect dots of the top row belonging to the same part from left to right. Do the same in the bottom row. Parts which contain both top and bottom row dots will be connected via the respective leftmost dots. Parts connecting top and bottom row are often called \textit{propagating parts} in the literature. The number $\#_p(d)$ of propagating parts of a diagram $d$ is called \textit{propagating number}. We call the actual line connecting a top and a bottom row dot \textit{propagating line}.  We denote the top row of a diagram $d$ by $\top(d)$, its bottom row by $\bottom(d)$ and the permutation induced by the propagating lines by $\Pi(d) \in \Sigma_{\#_p(d)}$. Note that multiplication of diagrams cannot increase the propagating number, since a propagating part of $x\cdot y$ connects $\top(x)$ to $\bottom(y)$ via $\bottom(x)=\top(y)$, hence $\#_p(x\cdot y) \leq \min \{\#_p(x),\#_p(y)\}$. The unit element of $P_k(r,\delta)$ is given by the set partition $\{\{1,1'\},\{2,2'\},...,\{r,r'\}\} = \begin{minipage}{3.5cm}
\xysmall{\bullet \tra[d] & \bullet \tra[d] & ... & \bullet \tra[d] \\
\bullet & \bullet & ... & \bullet} 
\end{minipage}$.

A diagram consisting of only one row with $r$ dots and arbitrary connections is called \textit{partial diagram}. We have to distinguish certain parts from others; we say they are \textit{labelled} and write the dots as empty circles $\circ$ instead of dots $\bullet$. When we complete a partial diagram to a full diagram with two rows of dots, the labelled parts become propagating.  We count the labelled parts from left to right, according to the leftmost dot of the part. 
Let $V_n$ be the vector space with basis all partial diagrams with exactly $n$ labelled parts (and possibly further unlabelled parts). 
For example, $\xymatrixcolsep{10pt}\xymatrix{\bullet \arcu[rr] & \circ & \bullet \tra[r] & \bullet & \circ \tra[r] & \circ & \bullet}$ is a basis element of $V_2$ in case $r=7$; the labelled singleton $\circ$ is the first labelled part, the part $\circ - \circ$ is the second.

\subsection{Structural Properties}\label{subsec: PA structure}
The group algebra $k\Sigma_r$ is a unitary subalgebra of $P_k(r,\delta)$, where a permutation $\pi \in \Sigma_r$ corresponds to a diagram where all parts are of size $2$ and propagating, i.e. each dot $i$ of the top row is connected to exactly one dot $\pi(i)$ of the bottom row. 
For $l < r$, we have different embeddings of $k\Sigma_l$ into $P_k(r,\delta)$ 
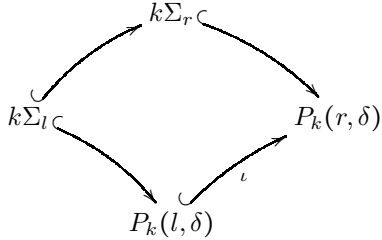
$$\xymatrix{& k\Sigma_r \arcu@{^{(}->}[dr] \\
k\Sigma_l \arcu@{^{(}->}[ur] \arcu@{^{(}->}[dr] && P_k(r,\delta)\\ 
& P_k(l,\delta) \arcu@{^{(}->}_-\iota[ur]}$$
where the inclusion $\iota: P_k(l,\delta) \hookrightarrow P_k(r,\delta)$ is given as follows. Let $b$ be a diagram in $P_k(l,\delta)$. Add dots $l+1,...,r$ to the top row and dots $(l+1)',...,r'$ to the bottom row of $b$ and attach these new dots to the $l$\th\hspace{-0.5ex}, respectively $l'$\th\hspace{-0.5ex}, dot of $b$. Throughout this article, we use the embedding $\hat{\iota}: k\Sigma_l \hookrightarrow P_k(l,\delta) \overset{\iota}{\hookrightarrow} P_k(r,\delta)$. Restriction to a (smaller) group algebra will also be via this embedding. 
\\
 
Xi showed that $P_k(r,\delta)$ is cellular by considering it as an iterated inflation of group algebras of symmetric groups. 

\begin{thm}[{\hspace*{-0.7ex}\cite[Theorem 4.1]{Xi}}]\label{thm: Xi}
The partition algebra $P_k(r,\delta)$ is cellular. More precisely, it is an iterated inflation of the form $\bigoplus\limits_{n=0}^r k\Sigma_n \otimes_k V_n \otimes_k V_n$, with respect to the involution $i$ turning a diagram upside down. \end{thm} 

The cell modules are called \emph{standard modules} in \cite{Xi}.
In \cite{HHKP}, the partition algebra (with $\delta \neq 0$) is one of the main examples for cellularly stratified algebras. For the cellularly stratified structure, we need the existence of idempotents $e_n=1_{\Sigma_n}\otimes u_n \otimes v_n$ such that $e_ne_m=e_m=e_me_n$ for $m\leq n$.\\

From now on, let $\delta \neq 0$ and set 
$$e_0 := \frac{1}{\delta} \cdot \begin{minipage}[c]{4cm} \xysmall{ \bullet^1 \tra[r]  & \bullet \tra[r] & ...  \tra[r]  & \bullet \tra[r] & \bullet^r \\  \bullet_{1'} \tra[r] & \bullet \tra[r] & ...  \tra[r]  & \bullet \tra[r] & \bullet_{r'} } \end{minipage} \text{,} \quad e_n := \begin{minipage}[c]{4cm} \xysmall{\bullet^1 \tra[d] & ... & \bullet \tra[d] & \bullet^n \tra[r] \tra[d] &  ...  \tra[r]  & \bullet^r \\ \bullet_{1'} & ... & \bullet & \bullet_{n'} \tra[r] &  ...  \tra[r] & \bullet_{r'}} \end{minipage} \text{ for } n\geq 1.$$

\begin{thm}[{\hspace*{-0.7ex}\cite[Proposition 2.6]{HHKP}}]\label{thm: HHKP}
The partition algebra $P_k(r,\delta)$ is cellularly stratified with stratification data $(k,V_0,k,V_1,k\Sigma_2,V_2,...,k\Sigma_r,V_r)$ and idempotents $e_n$ for all parameters $\delta \in k\setminus \{0\}$. 
\end{thm}

Intuitively, a cellular algebra is cellularly stratified if there is a chain of two-sided ideals ${0}= J_0 \subseteq J_1 \subseteq ... \subseteq J_r =A$ such that each subquotient $J_l/J_{l-1}$ is a non-unital algebra of the form $B_l \otimes V_l \otimes V_l$ for some smaller cellular algebra $B_l$ and vector space $V_l$. 

As a consequence of the cellularly stratified structure, we have that $k\Sigma_r$ is also a quotient of $P_k(r,\delta)$ by the ideal generated by all diagrams $d$ with $\#_p(d) \leq r-1$, i.e. $k\Sigma_r \simeq A/(Ae_{r-1}A)$, where $A=P_k(r,\delta)$ for any $\delta \neq 0$. \\

From now on, let $A:=P_k(r,\delta)$ and $\delta \neq 0$. By abuse of notation, we write $e_n(A/J_{n-1})$ for the $(e_nAe_n,A)$-bimodule $e_nA/e_nJ_{n-1}$, where $J_{n-1} = Ae_{n-1}A$ is the two-sided ideal generated by all diagrams $d$ with $\#_p(d) \leq n-1$. 
We regard $e_nA$ as left $k\Sigma_n$-module via the embedding\footnote{$e_nAe_n$ is the image of the embedding $\iota: P_k(n,\delta) \hookrightarrow P_k(r,\delta)$, hence $e_nAe_n \simeq P_k(n,\delta)$.} $k\Sigma_n \hookrightarrow e_nAe_n \simeq P_k(n,\delta)$. 

The group algebra $k\Sigma_n$ is cellular; we choose as cell modules the dual Specht modules $S_\lambda$. It follows from \cite[Lemma 3.4 and Proposition 4.2]{HHKP} that the cell modules of $A$ are of the form $(A/J_{n-1})e_n \tensorover{e_nAe_n} S_\nu$ for $0 \leq n \leq r$ and $\nu$ a partition of $n$, where $J_{-1}:= 0$. This coincides with the more combinatorial definition in \cite{Xi}. 
Let 
$$\mathcal{F}_n(S)=\left\lbrace M \in k\Sigma_n-\mod \text{ } \middle| \begin{aligned} M=M_s \supset M_{s-1} \supset ... \supset M_1 \supset M_0=0, \\ M_i/M_{i-1} \simeq S_{\lambda_i} \text{ for some partition } \lambda_i \text{ of } n \end{aligned} \right\rbrace$$ denote the category of $k\Sigma_n$-modules admitting a dual Specht filtration. 

\section{The bimodule $e_l(A/J_{n-1})e_n$}\label{sec: restriction cell modules}

This section is dedicated to the proof of Theorem \ref{TMR}, which gives a structural analysis of the $(k\Sigma_l,k\Sigma_n)$-bimodule $e_l(A/J_{n-1})e_n$. Note that the bimodule is zero if $n > l$, since in this case $e_l \in J_{n-1}$. Hence we always assume $n \leq l$.

Most of the notation for Theorem \ref{TMR} is defined in Subsection \ref{subsec: notation}. We prove Theorem \ref{TMR} in separate lemmas and propositions. 

\begin{mythm}\label{TMR}
Let $A= P_k(r,\delta)$ with $\delta \neq 0$ and let $0\leq n \leq l \leq r$. The following holds for the $(k\Sigma_l,k\Sigma_n)$-bimodule $e_l(A/J_{n-1})e_n$. 
\begin{enumerate}
\item\label{TMR decomposition} There is a bimodule decomposition $e_l(A/J_{n-1})e_n \simeq \bigoplus\limits_{v \in V_n^l/_\sim} U_v$, indexed by a set of equivalence classes of partial diagrams.
\item\label{TMR isomorphisms} Fix $v \in V_n^l/_\sim$ and let $l_1$ be the amount of labelled dots and $l_2$ the amount of unlabelled dots in $v$. Let $\prod_\alpha \subset \Sigma_{l_1}$ be the stabilizer of the labelled dots and let $\prod_\beta \subset \Sigma_{l_2}$ be the stabilizer of the unlabelled dots. Then $$\begin{aligned} U_v  \simeq & \, k\Sigma_l \tensorover{k(\prod_\alpha \times \prod_\beta)} k\Sigma_n  \\  \simeq & \, k\Sigma_l \tensorover{k\Sigma_{(l_1,l_2)}} \left((k\Sigma_{l_1} \tensorover{k\prod_\alpha} k\Sigma_n) \boxtimes (k\Sigma_{l_2} \tensorover{k\prod_\beta} k)\right) \end{aligned}$$ as $(k\Sigma_l,k\Sigma_n)$-bimodules.
\item\label{TMR labelled} $k\Sigma_{l_1} \tensorover{k\prod_\alpha} k\Sigma_n$ is isomorphic to $$k\Sigma_{l_1} 
\tensorover{k\Sigma_{(\alpha_1,2\alpha_2,...,s\alpha_s)}} \left( (k\Sigma_{\alpha_1} \tensorover{k\Sigma_{(1^{\alpha_1})}} k) \boxtimes ... \boxtimes (k\Sigma_{s\alpha_s} \tensorover{k\Sigma_{(s^{\alpha_s})}} k)\right) \tensorover{k\Sigma_\alpha} k\Sigma_n$$ as a $(k\Sigma_{l_1},k\Sigma_n)$-bimodule and to a direct sum of copies of $$k\Sigma_{l_1} \tensorover{k\Sigma_{(\alpha_1,2\alpha_2,...,s\alpha_s)}} \left((k\Sigma_{\alpha_1} \tensorover{k\Sigma_{(1^{\alpha_1})}} k) \boxtimes ... \boxtimes (k\Sigma_{s\alpha_s} \tensorover{k\Sigma_{(s^{\alpha_s})}} k)\right)$$ as a left $k\Sigma_{l_1}$-module. 
\item\label{TMR unlabelled} $k\Sigma_{l_2} \tensorover{k\prod_\beta} k$ is isomorphic to $$k\Sigma_{l_2} \tensorover{k\Sigma_{(\beta_1,2\beta_2,...,t\beta_t)}} \left(k \boxtimes (k\Sigma_{2\beta_2} \tensorover{k(\Sigma_2\wr\Sigma_{\beta_2})} k) \boxtimes ...\boxtimes (k\Sigma_{t\beta_t} \tensorover{k(\Sigma_t\wr\Sigma_{\beta_t})} k)\right)$$ as a left $k\Sigma_{l_2}$-module.
\item\label{TMR Specht} The left $k\Sigma_l$-module $e_l(A/J_{n-1})e_n$ admits a dual Specht filtration if $\text{\textnormal{char}}k = 0$ or $\text{\textnormal{char}}k > \lfloor \frac{l-n}{3} \rfloor$.
\end{enumerate}
\end{mythm}

\subsection{Notation and decomposition of $\boldsymbol{e_l(A/J_{n-1})e_n}$}\label{subsec: notation}
If $\lambda = (\lambda_1,\lambda_2,...,\lambda_s)$ is a composition of $r$, we denote the Young subgroup $\Sigma_{\lambda_1} \times \Sigma_{\lambda_2} \times ... \times \Sigma_{\lambda_s}$ by $\Sigma_\lambda$ and if $\lambda_i = \lambda_{i+1} = ... = \lambda_{i+j}$, we write $\lambda = (\lambda_1,...,\lambda_{i-1},\lambda_i^j,\lambda_{i+j+1},...,\lambda_s)$. 

Let $0 \leq n \leq l \leq r$, set $J_{-1}:= 0$ and let $V_n^l$ be the subspace of $V_n$ generated by all partial diagrams with $n$ labelled parts, where the last $r-l+1$ dots lie in the same part. 

We adopt the convention that these last $r-l+1$ dots of an element in $V_n^l$ are counted as just one dot. 
Let $v,w \in V_n^l$. We say that $v$ is equivalent to $w$, $v \sim w$, if and only if there is a $\pi \in \Sigma_l \subset e_lAe_l$ such that $\pi v = w$, where $\pi v$ is defined as follows. Write the diagram $\pi$ on top of $v$ and identify $\bottom(\pi)$ with $v$. Then $\pi v$ is the top row of this diagram, where a part is labelled if and only if it contains at least one labelled dot. 
In terms of diagrams, this means that $v$ and $w$ are equivalent, if and only if, for each size, the number of labelled parts and the number of unlabelled parts of $v$ and $w$ coincide. Recall that the last $r-l+1$ dots count as one. 

\begin{ex} \label{Ex: action of Sigma on V}
Let $r=7, l=6, n=2, \pi = (56) \in \Sigma_6$ and 
$$v = \xysmall{\circ \arcu[rr] & \bullet \arcu[rr] & \circ & \bullet \arcu[rr] & \circ & \bullet \tra[r] & \bullet }$$ Then $$ \begin{aligned}
\pi v & = \top \left(
\begin{minipage}[c]{4cm}
\xysmall{\bullet \tra[d] & \bullet \tra[d] & \bullet \tra[d] &\bullet \tra[d] &\bullet \tra[dr] &\bullet \tra[dl] \tra[r] &\bullet  \\
\circ \arcu[rr] & \bullet \arcu[rr] & \circ & \bullet \arcu[rr] & \circ & \bullet \tra[r] & \bullet }
\end{minipage}
\right) 
 & = \xysmall{\circ \arcu[rr] & \bullet \arcu[rr] & \circ & \bullet \tra[r] & \bullet & \circ \tra[r] & \circ }
 \end{aligned}$$ 
\end{ex}

For $v \in V_n^l$, we define $d_v$ to be the diagram with $\top(d_v)=v$, $\bottom(d_v)=\bottom(e_n)$ and $\Pi(d_v)=1_{k\Sigma_n}$. 
Let $b \in e_l(A/J_{n-1})e_n$ be a diagram with $\top(b) \sim v$. By definition, there is a $\pi \in \Sigma_l$ such that $\top(b)=\pi v$. Then $b = \pi d_v \Pi(\pi d_v)^{-1}\Pi(b)$.  Let $U_v$ be the $(k\Sigma_l,k\Sigma_n)$-bimodule generated by $d_v$. 
The following example explains how to write any diagram in $U_v$ in the form $\tau d_v \eta$. 
\begin{ex} Let $r=7,l=6,n=3, v=\xysmall{\circ \arcu[rr] & \circ \arcu[rr] & \circ & \circ \arcu[rr] & \circ & \circ \tra[r] & \circ}$ and $$b=\begin{minipage}[c]{4cm}
\xysmall{\bullet \tra[r] \tra[drr] & \bullet & \bullet \tra[dll] & \bullet \tra[r] \tra[dll] & \bullet \tra[r] & \bullet \tra[r] & \bullet \\ 
\bullet & \bullet & \bullet \tra[r] & \bullet \tra[r] &\bullet \tra[r] &\bullet \tra[r] & \bullet}\end{minipage}$$
Then $\top(b)=(2354)v$ and $\Pi(b)=(132)$. In particular, $$b=(2354)d_v\Pi((2354)d_v)^{-1}\Pi(b).$$
\end{ex}

\begin{lem}[Theorem \ref{TMR}, part (\ref{TMR decomposition})]\label{lemma: decomposition eAe}
The $(k\Sigma_l,k\Sigma_n)$-bimodule $e_l(A/J_{n-1})e_n$ decomposes into a direct sum $\bigoplus\limits_{v \in V_n^l/_\sim} U_v$. 
\end{lem}

\begin{proof}
Any diagram in $e_l(A/J_{n-1})e_n$ with top row in the equivalence class of $v$ equals $\tau d_v \eta$ for some $\tau \in \Sigma_l, \eta \in \Sigma_n$ by the remark above. 
 For $w \in V_n^l$ with $w \nsim v$, we have $U_w \cap U_v = \{0\}$
because a diagram in the intersection would have top row equivalent to $v$ and to $w$ simultaneously. Therefore, every diagram in $e_l(A/J_{n-1})e_n$ lies in exactly one of the sets $U_v$ and every diagram of $U_v$ is a diagram in $e_l(A/J_{n-1})e_n$ by definition. 
\end{proof}

\begin{rk} Let $p(m)$ denote the number of integer partitions of the natural number $m$ and let $p(m,k)$ denote the number of integer partitions of $m$ into $k$ parts. Set $p(0)=1$. Let $v \in V_n^l$.  
Assume that $v$ has $m$ labelled dots, where $n \leq m \leq l$, arranged into $n$ parts. Then $l-m$ dots are unlabelled. There are $p(m,n)p(l-m)$ possibilities for such a partial diagram $v$, up to equivalence. Hence $|V_n^l/_\sim| = \sum\limits_{m=n}^l p(m,n)p(l-m)$.  
\end{rk}

For the remainder of this section, we fix a partial diagram $v \in V_n^l$ and set $d:=d_v$. Let $\alpha_i$ be the number of labelled parts of size $i$ and $\beta_i$ the number of unlabelled parts of size $i$ of $v$, where again the last $r-l+1$ dots count as one dot. Then $\sum\limits_i (\alpha_i \cdot i) + \sum\limits_i (\beta_i \cdot i) = l$ and $\sum\limits_i \alpha_i = n$. Without loss of generality, assume that the parts of $v$ are ordered as follows. The labelled parts are on the left hand side, the unlabelled parts on the right hand side. The parts are then ordered increasingly from left to right. 
Let $\mathcal{S}_i^j \subseteq \{1,...,l\}$ be the set of dots of $v$ belonging to the $j$\th labelled part of size $i$ and let $\mathcal{T}_i^j \subseteq \{1,...,l\}$ be the set of dots of $v$ belonging to the $j$\th unlabelled part of size $i$. Then $\prod_\alpha:= \prod\limits_{{i \geq 1,  \alpha_i \neq 0}}((\Sigma_{\mathcal{S}_i^1} \times ... \times \Sigma_{\mathcal{S}_i^{\alpha_i}}) \rtimes \Sigma_{\alpha_i}) \subset \Sigma_l$ is the stabilizer subgroup of $\Sigma_l$ which stabilizes exactly the labelled parts of $v$. Similarly, $\prod_\beta:= \prod\limits_{{i \geq 1, \beta_i \neq 0}}((\Sigma_{\mathcal{T}_i^1} \times ... \times \Sigma_{\mathcal{T}_i^{\beta_i}}) \rtimes \Sigma_{\beta_i}) \subset \Sigma_l$ is the stabilizer subgroup of $\Sigma_l$ which stabilizes exactly the unlabelled parts of $v$. In particular, $\prod_\beta$ stabilizes $d$, while $\prod_\alpha$ can rearrange the propagating lines of $d$. Note that $\prod_\alpha \simeq \prod\limits_{i \geq 1, \alpha_i \neq 0} (\Sigma_i \wr \Sigma_{\alpha_i})$
 and $\prod_\beta \simeq \prod\limits_{i \geq 1, \beta_i \neq 0} (\Sigma_i \wr \Sigma_{\beta_i})$, where $\wr$ denotes the wreath product.
 
\begin{ex}
Let $r=12, l=11, n=3$ and $v = \xymatrixcolsep{4pt}\xymatrix{\circ \tra[r] & \circ & \circ \tra[r] & \circ & \circ \tra[r] & \circ & \bullet & \bullet \tra[r] & \bullet & \bullet \tra[r] & \bullet \tra[r] & \bullet}.$ Then $\alpha=(0,3), \beta=(1,2)$ and \begin{center}$\prod_\alpha = (\Sigma_{\{1,2\}} \times \Sigma_{\{3,4\}} \times \Sigma_{\{5,6\}}) \rtimes \Sigma_3 \simeq \Sigma_2 \wr \Sigma_3$\end{center} and \begin{center}$\prod_\beta = \Sigma_{\{7\}} \times ((\Sigma_{\{8,9\}} \times \Sigma_{\{10,11\}}) \rtimes \Sigma_2) \simeq \Sigma_1 \times (\Sigma_2 \wr \Sigma_2).$\end{center}
 \end{ex}

\subsection{The summands $\boldsymbol{U_v}$}\label{subsec: U as tensor product}
Consider the $(k\Sigma_l,k\Sigma_n)$-bimodule $k\Sigma_l \tensorover{k\prod_\alpha \times k\prod_\beta} k\Sigma_n$, where $\prod_\beta$ acts trivially on $k\Sigma_n$ and the action of $\prod_\alpha$ on $k\Sigma_n$ is given by $\zeta \cdot \eta := \Pi(\zeta d)\eta$ for $\zeta \in \prod_\alpha, \eta \in \Sigma_n$, i.e. $\prod_\alpha$ acts on $\Sigma_n$ via the canonical epimorphism $\prod_\alpha \twoheadrightarrow \Sigma_{\alpha_1} \times ... \times \Sigma_{\alpha_s}$. We have $\top(\zeta d) = \top(d)$, so $\zeta d = d \Pi(\zeta d)$ for $\zeta \in \prod_\alpha$.    

\begin{lem}\label{lemma: psi}
The map 	
$$ \begin{aligned}
\psi: \, k\Sigma_l & \tensorover{k\prod_\alpha \times k\prod_\beta} k\Sigma_n  & \longrightarrow & \text{ } U_v \\
& \quad \tau \otimes \eta & \longmapsto & \text{ } \tau d \eta \end{aligned}$$
is an isomorphism of $(k\Sigma_l,k\Sigma_n)$-bimodules. 
\end{lem}

\begin{proof}
Let $x \in \prod_\alpha$ and $y \in \prod_\beta$. Then $yd=d$ and $xd=d\Pi(xd)$. Consider the map $\Psi:k\Sigma_l \times k\Sigma_n \to U_v$ given by $ \Psi(\tau,\eta)=\tau d \eta$. Then $\Psi(\tau xy,\eta)=\tau xyd\eta = \tau d\Pi(xd)\eta = \Psi(\tau,\Pi(xd)\eta) = \Psi(\tau,xy\cdot \eta)$. This shows that $\psi$ is well-defined. 

Let $\tau,\tau' \in \Sigma_l$ and $\eta,\eta' \in \Sigma_n$. Then $\psi(\tau'\tau,\eta\eta')=\tau'\tau d \eta\eta' = \tau' \psi(\tau,\eta) \eta'$, so $\psi$ is a bimodule-homomorphism. 

The inverse map is given by
$$\begin{aligned}
\tilde{\psi}: \, & U_v & \longrightarrow & \text{ } k\Sigma_l \tensorover{k\prod_\alpha \times k\prod_\beta} k\Sigma_n &\\
& b & \longmapsto & \text{ } \tau \otimes \Pi(\tau d)^{-1} \Pi(b) & \text{ if } \top(b)=\tau v. \end{aligned}$$

We show that $\tilde{\psi}$ is well-defined, i.e. we show that it is independent of the choice of $\tau$. If $\top(b)=\tau_1 v = \tau_2 v$, there are $x \in \prod_\alpha$, $y \in \prod_\beta$ such that $\tau_1 = \tau_2 xy$. Then $$ \begin{aligned}
\tau_1 \otimes \Pi(\tau_1 d)^{-1}\Pi(b) & = \tau_2 xy \otimes \Pi(\tau_2 xyd)^{-1}\Pi(b)  \\
& \overset{(\dagger)}{=} \tau_2 \otimes \Pi(xd)\Pi(\tau_2 d\Pi(xd))^{-1}\Pi(b)  \\
& \overset{(\ast)}{=} \tau_2 \otimes \Pi(xd)(\Pi(\tau_2 d)\Pi(xd))^{-1}\Pi(b) \\
&= \tau_2 \otimes \Pi(xd)\Pi(xd)^{-1}\Pi(\tau_2 d)^{-1}\Pi(b) \\
&= \tau_2 \otimes \Pi(\tau_2 d)^{-1}\Pi(b). \end{aligned} $$

The equation $(\dagger)$ holds by definition of the action of $\prod_\alpha \times \prod_\beta$ on $k\Sigma_n$ and on $d$. The equation $(\ast)$ holds since the permutation $\Pi(a\eta)$ induced by the propagating lines of the product of a diagram $a \in Ae_n$ and a permutation diagram $\eta \in \Sigma_n$ is read from top to bottom, so it does not matter whether we consider $\eta$ as (bottom) part of the diagram $a\eta$ or as an independent diagram, hence $\Pi(a\eta) = \Pi(a)\eta$. 

Let $b \in U_v$ be a diagram with $\top(b)=\tau v$ for some $\tau \in \Sigma_l$ and let $\eta \in \Sigma_n$. Then $\psi\tilde{\psi}(b)=\psi(\tau \otimes \Pi(\tau d)^{-1}\Pi(b))=\tau d \Pi(\tau d)^{-1}\Pi(b)=b$ and $\tilde{\psi}\psi(\tau \otimes \eta) = \tilde{\psi}(\tau d \eta) = \tau \otimes \Pi(\tau d)^{-1}\Pi(\tau d \eta) = \tau \otimes \eta$, so $\tilde{\psi}$ is the inverse of $\psi$. 
\end{proof}

Set $l_1:=\sum\limits_i\alpha_i\cdot i$ and $l_2:=\sum\limits_i \beta_i\cdot i$, so $l=l_1+l_2$, $\prod_\alpha \subset \Sigma_{l_1} $ and define $\Sigma_{l_2}:=\Sigma_{\{l_1+1,...,l\}} = \Sigma_{\bigcup\limits_{i,j} \mathcal{T}_i^j} \supset \prod_\beta$.   Fix coset representatives $\omega_1,...,\omega_t$ of $k\Sigma_l/k\Sigma_{(l_1,l_2)}$\label{cosets omega}. Denote by $X \boxtimes Y \in k\Sigma_{(l_1,l_2)}-\mod$ the exterior tensor product of $X \in k\Sigma_{l_1}-\mod$ and $Y \in k\Sigma_{l_2}-\mod$ given by $$(\tau_1,\tau_2)\cdot(x \boxtimes y)= \tau_1 x \boxtimes \tau_2 y$$ for $\tau_1 \in \Sigma_{l_1}, \tau_2 \in \Sigma_{l_2},x \in X, y \in Y$.
 
Consider the $(k\Sigma_l,k\Sigma_n)$-bimodule $k\Sigma_l \tensorover{k\Sigma_{(l_1,l_2)}} \left((k\Sigma_{l_1} \tensorover{k\prod_\alpha} k\Sigma_n) \boxtimes (k\Sigma_{l_2} \tensorover{k\prod_\beta} k)\right)$ with right $k\Sigma_n$-module structure given by 
$$(\omega \otimes ((\tau_1 \otimes \eta) \boxtimes (\tau_2 \otimes 1)))\cdot \eta':=\omega \otimes ((\tau_1 \otimes \eta\eta') \boxtimes (\tau_2 \otimes 1))$$ for $\omega \otimes ((\tau_1 \otimes \eta) \boxtimes (\tau_2 \otimes 1)) \in k\Sigma_l \tensorover{k\Sigma_{(l_1,l_2)}} \left((k\Sigma_{l_1} \tensorover{k\prod_\alpha} k\Sigma_n) \boxtimes (k\Sigma_{l_2} \tensorover{k\prod_\beta} k)\right)$ and $\eta' \in \Sigma_n$, i.e. $\Sigma_n$ acts regularly on $k\Sigma_n$ and trivially on $k$. 

\begin{lem}\label{lemma: theta}
$k\Sigma_l \tensorover{k\prod_\alpha \times k\prod_\beta} k\Sigma_n$ and $k\Sigma_l \tensorover{k\Sigma_{(l_1,l_2)}} \left((k\Sigma_{l_1} \tensorover{k\prod_\alpha} k\Sigma_n) \boxtimes (k\Sigma_{l_2} \tensorover{k\prod_\beta} k)\right)$  are isomorphic as $(k\Sigma_l,k\Sigma_n)$-bimodules. 
\end{lem}

\begin{proof}
Define $$\begin{aligned}
\Theta: \, & k\Sigma_l \times  k\Sigma_n  &\longrightarrow &\text{ } k\Sigma_l \tensorover{k\Sigma_{(l_1,l_2)}} \left((k\Sigma_{l_1} \tensorover{k\prod_\alpha} k\Sigma_n) \boxtimes (k\Sigma_{l_2} \tensorover{k\prod_\beta} k)\right) &\\
&(\tau,\eta) & \longmapsto &\text{ } \omega_i \otimes ((\tau_1 \otimes \eta) \boxtimes (\tau_2 \otimes 1)) & \text{if } \tau=\omega_i\tau_1\tau_2 \end{aligned}$$
where $\tau_1 \in \Sigma_{l_1}$ and $\tau_2 \in \Sigma_{l_2}$. Let $x \in \prod_\alpha$, $y \in \prod_\beta$. 
Then $x \in \Sigma_{l_1}\times \{0\} \subset \Sigma_l$, so $\tau_2x=x\tau_2$ for $\tau_2 \in \Sigma_{l_2}$. Thus, 
$\Theta(\tau xy,\eta) 
= \Theta(\omega_i\tau_1\tau_2xy,\eta) 
= \Theta(\omega_i\tau_1x\tau_2y,\eta) 
= \omega_i \otimes ((\tau_1x \otimes \eta) \boxtimes (\tau_2y \otimes 1)) 
= \omega_i \otimes ((\tau_1 \otimes \Pi(xd)\eta) \boxtimes (\tau_2 \otimes 1)) 
= \Theta(\tau, \Pi(xd)\eta) 
= \Theta(\tau, xy\cdot\eta)$. 
Hence, the map 
$$\begin{aligned} 
\theta: \, & k\Sigma_l \tensorover{k\prod_\alpha \times k\prod_\beta} k\Sigma_n  & \longrightarrow & \text{ }  k\Sigma_l \tensorover{k\Sigma_{(l_1,l_2)}} \left((k\Sigma_{l_1} \tensorover{k\prod_\alpha} k\Sigma_n) \boxtimes (k\Sigma_{l_2} \tensorover{k\prod_\beta} k)\right)&\\
&\tau \otimes \eta & \longmapsto &\text{ } \omega_i \otimes ((\tau_1 \otimes \eta) \boxtimes (\tau_2 \otimes 1))  & \hspace*{-0.55cm}\text{if } \tau = \omega_i\tau_1\tau_2 \end{aligned}$$ is well-defined.

Let $\tau'\in \Sigma_l$ and let $j \in \{1,...,t\}$, $\tau_1' \in \Sigma_{l_1}$, $\tau_2' \in \Sigma_{l_2}$ such that $\tau'\omega_i=\omega_j\tau_1'\tau_2'$. Let $\eta'\in \Sigma_n$. Then 
$$\begin{aligned}
\theta(\tau'\tau\otimes\eta\eta')&= \theta(\omega_j\tau_1'\tau_2'\tau_1\tau_2 \otimes \eta\eta') &  \\
& =\theta(\omega_j\tau_1'\tau_1\tau_2'\tau_2 \otimes \eta\eta') & \text{since } \tau_1 \in \Sigma_{l_1} \text{ and } \tau_2' \in \Sigma_{l_2} \text{ commute}   \\
&=\omega_j \otimes ((\tau_1'\tau_1 \otimes \eta\eta') \boxtimes (\tau_2'\tau_2 \otimes 1)) &\text{by definition of the map } \theta\\
&= \omega_j \otimes (\tau_1'\tau_2'((\tau_1\otimes \eta\eta') \boxtimes (\tau_2 \otimes 1))) & \text{by definition of the exterior tensor product}\\
&= \omega_j \tau_1'\tau_2' \otimes ((\tau_1 \otimes \eta\eta') \boxtimes (\tau_2 \otimes 1)) &\\
&= \tau'\omega_i \otimes ((\tau_1 \otimes \eta\eta') \boxtimes (\tau_2 \otimes 1)) & \\
&= \tau'(\omega_i \otimes ((\tau_1 \otimes \eta) \boxtimes (\tau_2 \otimes 1))) \eta' &\text{by the right } k\Sigma_n-\text{action}\\
&= \tau' \theta(\tau\otimes \eta)\eta' &\text{by definition of the map } \theta
\end{aligned}$$ 
 so $\theta$ is a homomorphism of $(k\Sigma_l,k\Sigma_n)$-bimodules. 

The inverse is given by $$\begin{aligned}
\theta^{-1}: \, & k\Sigma_l \tensorover{k\Sigma_{(l_1,l_2)}} \left((k\Sigma_{l_1} \tensorover{k\prod_\alpha} k\Sigma_n) \boxtimes (k\Sigma_{l_2} \tensorover{k\prod_\beta} k)\right) & \longrightarrow &\text{ } k\Sigma_l \tensorover{k\prod_\alpha \times k\prod_\beta} k\Sigma_n \\ 
& \tau \otimes ((\vartheta \otimes \eta) \boxtimes (\upsilon \otimes 1)) & \longmapsto &\text{ } (\tau\vartheta \upsilon \otimes \eta) \end{aligned}\vspace*{-0.5cm}$$ 
\end{proof}

This concludes the proof of part (\ref{TMR isomorphisms}) of Theorem \ref{TMR}.

\subsection{The Foulkes module}\label{subsec: Foulkes module}
We start with a deeper study of the module $k\Sigma_{l_2} \tensorover{k\prod_\beta} k$ in this section and continue with a study of the more complex module $k\Sigma_{l_1} \tensorover{k\prod_\alpha} k\Sigma_n$ in the next section. 
Let $t$ be the maximal size of an unlabelled part of $v$, keeping in mind that the last $r-l+1$ dots count as one, and
set $\Sigma_{\gamma_i}:= \Sigma_{\bigcup\limits_j \mathcal{T}_i^j}$ for $i=1,...,t$ and $\Sigma_{\gamma}:= \prod\limits_i \Sigma_{\gamma_i}$. Then $\Sigma_\gamma \simeq \Sigma_{(\beta_1,2\beta_2,...,t\beta_t)}$ and $\prod_\beta \subset \Sigma_{\gamma} \subset \Sigma_{l_2}$.

\begin{lem}[Theorem \ref{TMR}, part (\ref{TMR unlabelled})]\label{lemma: tensor product of foulkes modules}
There is an isomorphism $$ k\Sigma_{l_2} \tensorover{k\prod_\beta} k \simeq k\Sigma_{l_2} \tensorover{k\Sigma_{\gamma}} \left(k \boxtimes (k\Sigma_{2\beta_2} \tensorover{k(\Sigma_2\wr\Sigma_{\beta_2})} k) \boxtimes ...\boxtimes (k\Sigma_{t\beta_t} \tensorover{k(\Sigma_t\wr\Sigma_{\beta_t})} k)\right)$$ of left $k\Sigma_{l_2}$-modules.
\end{lem}

\begin{proof}
Let $\epsilon_1,...,\epsilon_u$ be coset representatives of $\Sigma_{l_2}/\Sigma_{\gamma}$ and let $\tau=\epsilon_i \tau_1\tau_2...\tau_t$ with $\tau_j \in \Sigma_{\gamma_j}$. Then $\Sigma_{l_2} \tensorover{k\Pi_\beta} k  \ni \tau \otimes 1 = \epsilon_i\tau_1\tau_2...\tau_t \otimes 1 = \epsilon_i\tau_2...\tau_t \otimes 1$ since $\Pi_\beta = \Sigma_{\beta_1} \times (\prod\limits_{j\geq 2} (\Sigma_j \wr \Sigma_{\beta_j}))$ and $\tau_1 \in \Sigma_{\gamma_1}=\Sigma_{\beta_1}$. The assignment 
$$
\tau \otimes 1 \mapsto \epsilon_i \otimes (1 \boxtimes (\tau_2\otimes 1) \boxtimes ... \boxtimes (\tau_t \otimes 1)) $$ defines the isomorphism, like in Lemma \ref{lemma: theta}.
\end{proof}

The module $H^{(a^m)}:=k\Sigma_{am} \tensorover{k(\Sigma_a \wr \Sigma_m)} k$ is called the \emph{Foulkes module} for the parameters $a$ and $m$. 
If the characteristic of the field $k$ is strictly greater than $m$, or zero, the Foulkes module is isomorphic to a direct summand of the permutation module $M^{(a^m)}:= k\Sigma_{am} \tensorover{k\Sigma_{(a^m)}} k$, as mentioned in \cite{G}. We will give a proof of this statement in Lemma \ref{lemma: Foulkes module}. In smaller positive characteristic, this is not true. In fact, Giannelli shows in \cite[Theorem 1.1]{G} that for $0 < \text{\textnormal{char}}k \leq m$ there is a non-projective summand of $H^{(a^m)}$ which is not a Young module. In general, it is not known  whether or not a Foulkes module $H^{(a^m)}$ has a Specht filtration in the case $0 < \text{\textnormal{char}}k \leq m$ and $a > 3$. 
The case $a=2$ was solved in \cite{P} for arbitrary characteristic of the field.  Further special cases for Foulkes modules with a (dual) Specht filtration can be found in \cite[Theorems 2 and 3]{WildonMultiplicity}.

\begin{lem}\label{lemma: Foulkes module}
If $\text{\textnormal{char}} k = 0$ or $\text{\textnormal{char}} k > m$, the Foulkes module $$H^{(a^m)}=k\Sigma_{am} \tensorover{k(\Sigma_a \wr \Sigma_m)} k$$ is isomorphic to a direct summand of the permutation module $M^{(a^m)}$. 
\end{lem}

\begin{proof}
At the end of Subsection 3.1, we determined that the stabilizer of the $m$ unlabelled parts of size $a$ of a partial diagram is isomorphic to the wreath product $\Sigma_a \wr \Sigma_m$. The Foulkes module $k\Sigma_{am} \tensorover{k(\Sigma_a \wr \Sigma_m)} k$ has a vector space basis indexed by left cosets $\Sigma_ {am}/(\Sigma_a \wr \Sigma_m)$. In terms of partial diagrams, each coset in $\Sigma_{am}/(\Sigma_a \wr \Sigma_m)$ contains the information which dots belong to the same part\footnote{Any element in the coset $\pi(\Sigma_a\wr\Sigma_m)$ sends the diagram $w$ with $m$ parts of size $a$ sitting side by side to the same partial diagram $w'$ via left multiplication, i.e. $|\pi(\Sigma_a\wr\Sigma_m)w|=1$.}. Thus, $k\Sigma_{am} \tensorover{k(\Sigma_a \wr \Sigma_m)} k$ has a vector space basis of set partitions of the form $$\{\{x_1,...,x_a\},...,\{x_{(m-1)a+1},...,x_{ma}\}\}$$ with $x_i \in\{1,...,am\}, x_i \neq x_j$ for $i\neq j$.

Recall that the permutation module $M^{(a^m)}$ has a basis of $(a^m)$-tabloids. To show that $H^{(a^m)}$ is a summand of $M^{(a^m)}$, we show that there is a split epimorphism $\Phi: \, M^{(a^m)} \to H^{(a^m)}$ with right-inverse $\Psi$. For the definition of the right inverse we need an action of $\Sigma_m$ on $(a^m)$-tabloids which is given by permutation of the rows. We denote this action by $\ast$. The usual action of $\Sigma_{am}$ on $(a^m)$-tabloids is denoted by $\cdot$ as usual.

Define maps $\xymatrix{M^{(a^m)} \ar@<0.5ex>[r]^\Phi & H^{(a^m)} \ar@<0.5ex>[l]^\Psi}$ with 
$$\begin{aligned}
\begin{array}{ccc} \cline{1-3} x_1 & ... & x_a \\ \cline{1-3} & \vdots & \\ \cline{1-3} x_{(m-1)a+1} & ... & x_{ma} \\ \cline{1-3} \end{array}  & \overset{\Phi}{\longmapsto}  & \{\{x_1,...,x_a\},...,\{x_{(m-1)a+1},...,x_{ma}\}\} \\\\
\frac{1}{m!} \sum\limits_{\sigma \in \Sigma_m} \sigma \ast \begin{array}{ccc}
\cline{1-3} x_1 & ... & x_a \\ \cline{1-3} & \vdots & \\ \cline{1-3} x_{(m-1)a+1} & ... & x_{ma} \\ \cline{1-3}
\end{array} & \overset{\Psi}{\longmapsfrom} &
\{\{x_1,...,x_a\},...,\{x_{(m-1)a+1},...,x_{ma}\}\} 
\end{aligned}
$$

The $\ast$ and $\cdot$ actions commute: Let $\sigma \in \Sigma_m, \tau \in \Sigma_{am}$ and $x_i$ in row $k$ of the tabloid $x$. If $\sigma(k)=l$, then $x_i$ is in row $l$ of $\sigma \ast x$, so $x_{\tau(i)}$ is in row $l$ of $\tau\cdot (\sigma \ast x)$. On the other hand, $x_{\tau(i)}$ is in row $k$ of $\tau \cdot x$ and therefore in row $l$ of $\sigma \ast (\tau \cdot x)$.

We have $$\begin{aligned}
&\Phi \left(\tau \cdot \begin{array}{ccc} \cline{1-3} x_1 & ... & x_a \\ \cline{1-3} & \vdots & \\ \cline{1-3} x_{(m-1)a+1} & ... & x_{ma} \\ \cline{1-3} \end{array}\right) \\
= & \Phi \left( \begin{array}{ccc} \cline{1-3} x_{\tau(1)} & ... & x_{\tau(a)} \\ \cline{1-3} & \vdots & \\ \cline{1-3} x_{\tau((m-1)a+1)} & ... & x_{\tau(ma)} \\ \cline{1-3} \end{array} \right) \\
= & \{\{x_{\tau(1)},...,x_{\tau(a)}\},...,\{x_{\tau((m-1)a+1)},...,x_{\tau(ma)}\}\} \\
 = &\tau \cdot \{\{x_1,...,x_a\},...,\{x_{(m-1)a+1},...,x_{ma}\}\} \\
 = &\tau \Phi\left( \begin{array}{ccc} \cline{1-3} x_1 & ... & x_a \\ \cline{1-3} & \vdots & \\ \cline{1-3} x_{(m-1)a+1} & ... & x_{ma} \\ \cline{1-3} \end{array}\right) \end{aligned}$$ 
 and 
 $$\begin{aligned}
& \Psi\left(\tau \cdot \{\{x_1,...,x_a\},...,\{x_{(m-1)a+1},...,x_{ma}\}\}\right) \\
 = &\Psi\left(\{\{x_{\tau(1)},...,x_{\tau(a)}\},...,\{x_{\tau((m-1)a+1)},...,x_{\tau(ma)}\}\}\right) \\
 = & \frac{1}{m!}\sum\limits_{\sigma\in\Sigma_m} \sigma\ast \begin{array}{ccc} \cline{1-3} x_{\tau(1)} & ... & x_{\tau(a)} \\ \cline{1-3} & \vdots & \\ \cline{1-3} x_{\tau((m-1)a+1)} & ... & x_{\tau(ma)} \\ \cline{1-3} \end{array} \\
 = & \frac{1}{m!} \sum\limits_{\sigma\in\Sigma_m} \sigma\ast \left( \tau \cdot \begin{array}{ccc} \cline{1-3} x_1 & ... & x_a \\ \cline{1-3} & \vdots & \\ \cline{1-3} x_{(m-1)a+1} & ... & x_{ma} \\ \cline{1-3} \end{array}\right) \\
 = & \tau \cdot \left(\frac{1}{m!} \sum\limits_{\sigma\in\Sigma_m} \sigma\ast \begin{array}{ccc} \cline{1-3} x_1 & ... & x_a \\ \cline{1-3} & \vdots & \\ \cline{1-3} x_{(m-1)a+1} & ... & x_{ma} \\ \cline{1-3} \end{array}\right)\\
= & \tau \cdot \Psi\left(\{\{x_1,...,x_a\},...,\{x_{(m-1)a+1},...,x_{ma}\}\}\right).\end{aligned}$$
Hence, $\Psi$ and $\Phi$ are $k\Sigma_{am}$-module homomorphisms. $\Phi$ is surjective and $\Phi\Psi$ is the identity on $H^{(a^m)}$, so $\Phi$ is a split epimorphism.  
\end{proof}

\begin{cor}\label{cor: Foulkes Specht}
 If $\text{\textnormal{char}} k = 0$ or $\text{\textnormal{char}}k > m$, the indecomposable direct summands of the Foulkes module $H^{(a^m)}=k\Sigma_{am}\tensorover{k(\Sigma_a\wr\Sigma_m)} k$ are Young modules. In particular, $H^{(a^m)}$ is both Specht and dual Specht filtered. 
\end{cor}

\begin{cor}\label{cor: iterated Foulkes}
 If $\text{\textnormal{char}} k = 0$ or $\text{\textnormal{char}}k > \max \beta_i$, then
$k\Sigma_{l_2} \tensorover{k\prod_\beta} k \in \mathcal{F}_{l_2}(S)$.
\end{cor}

\begin{proof}
By Lemma \ref{lemma: tensor product of foulkes modules}, $k\Sigma_{l_2} \tensorover{k\prod_\beta} k$ is induced from an exterior tensor product of Foulkes modules. Corollary \ref{cor: Foulkes Specht} shows that the Foulkes modules are dual Specht filtered, provided the characteristic of the field is large enough. The characteristic-free version of the Littlewood-Richardson rule \cite{JP} then says that the exterior tensor product of Foulkes modules has a dual Specht filtration. 
\end{proof}

\subsection{Two-sided induction of permutation modules with bimodule structure}\label{subsec: permutation bimodule}
In this subsection, we examine the module $k\Sigma_{l_1} \tensorover{k\prod_\alpha} k\Sigma_n$ corresponding to the labelled parts of $v$. Recall that $l_1$ is the number of labelled dots in the top row of $d=d_v$ and that there are $\alpha_i$ labelled parts of size $i$. Let $s$ be the maximal size of a labelled part, then $l_1 = \sum\limits_{i=1}^s i\alpha_i$ and $n=\sum\limits_{i=1}^s \alpha_i$.

\begin{prop}\label{prop: decomposition of labelled parts into small diagrams} 
The $(k\Sigma_{l_1},k\Sigma_n)$-bimodule $k\Sigma_{l_1} \tensorover{k\prod_\alpha} k\Sigma_n$ is isomorphic to $$k\Sigma_{l_1} 
\tensorover{k\Sigma_{(\alpha_1,2\alpha_2,...,s\alpha_s)}} \left( (k\Sigma_{\alpha_1} \tensorover{k(\Sigma_1 \wr \Sigma_{\alpha_1})} k\Sigma_{\alpha_1}) \boxtimes ... \boxtimes (k\Sigma_{s\alpha_s} \tensorover{k(\Sigma_s \wr \Sigma_{\alpha_s})} k\Sigma_{\alpha_s} )\right) \tensorover{k\Sigma_\alpha} k\Sigma_n.$$
\end{prop}

\begin{figure}[h!]\caption{Maps for Proposition \ref{prop: decomposition of labelled parts into small diagrams}. The maps $\tilde{\psi}$ and $\theta$ are the maps from Lemmas 2 and 3, respectively. The remaining maps are defined in the course of the proof of the proposition.}\label{figure}
$$\xymatrix{
& k\Sigma_l\tensorover{k\Sigma_{(l_1,l_2)}}((k\Sigma_{l_1} \tensorover{k\prod_\alpha}k\Sigma_n)\boxtimes(k\Sigma_{l_2}\tensorover{k\prod_\beta}k)) \ar@{->>}[dl]_-{\rho_3}\\
\mathbf{k\Sigma_{l_1}\tensorover{k\prod_\alpha} k\Sigma_n} \ar[ddr]^-{\rho_1} & k\Sigma_l\tensorover{k(\prod_\alpha \times \prod_\beta)} k\Sigma_n \ar[u]^-{\wr}_-{\theta} \\
& U_v \ar[u]^-{\wr}_-{\tilde{\psi}}\\
& \mathbf{k\Sigma_{l_1}\tensorover{k\Sigma_{(\alpha_1,2\alpha_2,...,s\alpha_s)}}\left(\underset{i}{\boxtimes} (k\Sigma_{i\alpha_i} \tensorover{k(\Sigma_i \wr \Sigma_{\alpha_i})}k\Sigma_{\alpha_i})\right) \tensorover{k\Sigma_\alpha} k\Sigma_n} \ar[u]_-{\rho_2}}$$
\end{figure}

\begin{proof}
We prove this by giving explicit homomorphisms in both directions, as depicted in Figure \ref{figure}. 
The inverse map is given by $\rho_1^{-1}=\rho_3 \circ \theta \circ \tilde{\psi} \circ \rho_2$.
 
Fix coset representatives $\epsilon_1,...,\epsilon_p$ of $\Sigma_{l_1}/\Sigma_{(\alpha_1,2\alpha_2,...,s\alpha_s)}$ and $\eta_1,...,\eta_q$ of $\Sigma_\alpha\backslash \Sigma_n$.
Define a map $$\rho_1 :\, k\Sigma_{l_1} \tensorover{k\prod_\alpha} k\Sigma_n \longrightarrow k\Sigma_{l_1} \tensorover{k\Sigma_{(\alpha_1,2\alpha_2,...,s\alpha_s)}} \left( \underset{i}{\boxtimes} (k\Sigma_{i\alpha_i} \tensorover{k(\Sigma_i \wr \Sigma_{\alpha_i})} k\Sigma_{\alpha_i})\right) \tensorover{k\Sigma_\alpha} k\Sigma_n$$ by sending $\tau \otimes \sigma$ to $\epsilon \otimes (\underset{i}{\boxtimes} (\tau_i \otimes \sigma_i)) \otimes \eta$, if $\tau = \epsilon \tau_1\tau_2...\tau_s$ and $\sigma=\sigma_1\sigma_2...\sigma_s\eta$ with $\tau_i \in \Sigma_{i\alpha_i}$, $\sigma_i \in \Sigma_{\alpha_i}$ for $i=1,...,s$ and $\epsilon \in \{\epsilon_1,...,\epsilon_p\}$, $\eta \in \{\eta_1,...,\eta_q\}$. We show that this map is well-defined and leave the proof that it is a $(k\Sigma_{l_1},k\Sigma_n)$-bimodule homomorphism to the reader (works analoguosly to Lemma 3).
Let $\tau = \epsilon\tau_1...\tau_s \in \Sigma_{l_1}$ with $\epsilon \in \{\epsilon_1,...,\epsilon_p\}, \, \tau_i \in \Sigma_{i\alpha_i}$, and $\sigma = \sigma_1...\sigma_s\eta \in \Sigma_n$ with $\eta \in \{\eta_1,...,\eta_q\}, \, \sigma_i \in \Sigma_{\alpha_i}$. Let $\zeta = (\xi_{1,1},...,\xi_{1,{\alpha_1}};\zeta_1)...(\xi_{s,1},...,\xi_{s,{\alpha_s}};\zeta_s) \in \prod_\alpha$ with $\xi_{i,j} \in \Sigma_{\mathcal{S}_i^j}, \, \zeta_i \in \Sigma_{\alpha_i}$, in particular $(\xi_{i,1},...,\xi_{i,{\alpha_i}};\zeta_i)$ and $\tau_i$ are elements of the same symmetric group 
$\Sigma_{i\alpha_i}$.
Then $\tau\zeta = \epsilon\tau_1...\tau_s(\xi_{1,1},...,\xi_{1,{\alpha_1}};\zeta_1)...(\xi_{s,1},...,\xi_{s,{\alpha_s}};\zeta_s) = \epsilon\tau_1(\xi_{s,1},...,\xi_{s,{\alpha_s}};\zeta_s)\tau_2(\xi_{2,1},...,\xi_{2,{\alpha_2}};\zeta_2)...\tau_s(\xi_{s,1},...,\xi_{s,{\alpha_s}};\zeta_s)$ and $\zeta \cdot \sigma = \Pi(\zeta d)\sigma = \zeta_1\zeta_2...\zeta_s\sigma_1...\sigma_s\eta = \zeta_1\sigma_1\zeta_2\sigma_2...\zeta_s\sigma_s\eta$. Now we can calculate the images of $\tau \zeta \otimes \sigma$ and $\tau \otimes \Pi(\zeta d)\sigma$ under the map $\rho_1$ and compare them:  $$\rho_1(\tau\zeta \otimes \sigma) = \epsilon \otimes \left( \underset{i}{\boxtimes} (\tau_i(\xi_{i,1},...,\xi_{i,{\alpha_i}};\zeta_i) \otimes \sigma_i) \right) \otimes \eta$$ 
and $$\rho_1(\tau \otimes \Pi(\zeta d)\sigma) = \rho_1(\tau \otimes \zeta_1\sigma_1\zeta_2\sigma_2...\zeta_s\sigma_s\eta) = \epsilon \otimes \left(\underset{i}{\boxtimes} (\tau_i \otimes \zeta_i\sigma_i)\right) \otimes \eta.$$
Since $(\xi_{i,1},...,\xi_{i,{\alpha_i}};\zeta_i) \in \Sigma_i \wr \Sigma_{\alpha_i}$, we can move this element across the tensor to the right and $\Pi((\xi_{i,1},...,\xi_{i,{\alpha_i}};\zeta_i)d)=\zeta_i$, so $\rho_1(\tau\zeta \otimes \sigma) = \rho_1(\tau \otimes \Pi(\zeta d)\sigma)$ and the map is well-defined.

Next, we construct a map $$\rho_2:\, k\Sigma_{l_1} \tensorover{k\Sigma_{(\alpha_1,2\alpha_2,...,s\alpha_s)}} \left( \underset{i}{\boxtimes} (k\Sigma_{i\alpha_i} \tensorover{k(\Sigma_i \wr \Sigma_{\alpha_i})} k\Sigma_{\alpha_i})\right) \tensorover{k\Sigma_\alpha} k\Sigma_n \longrightarrow U_v,$$ which will be the crucial ingredient for the inverse of $\rho_1$. For this map, we need a special diagram $x$ which is obtained as follows. For $i=1,...,s$, we set $d_i$ to be the diagram in $P_k(i\alpha_i,\delta)e_{\alpha_i}$ with top row consisting of $\alpha_i$ parts of size $i$ sitting side by side and $\alpha_i$ non-crossing propagating lines. The bottom row consists of $\alpha_i-1$ singletons followed by a part of size $i\alpha_i-(\alpha_i-1)$. Then we have $s$ small diagrams $x_i := \tau_i d_i \sigma_i \in P_k(i\alpha_i,\delta)e_{\alpha_i}$, one for each $i \in \{1,...,s\}$, which we can combine to a big diagram $x' \in P_k(l_1,\delta)$ by simply writing the diagrams next to each other. To obtain the diagram $x \in P_k(r,\delta)e_n$, we now change the bottom row such that the first $\alpha_1$ dots belong to the first $\alpha_1$ dots of the bottom row of $x_1$, the following $\alpha_2$ dots belong to the first $\alpha_2$ dots of the bottom row of $x_2$ and so on. After we have rearranged the $\alpha_s$ dots of the bottom row of $x_s$, we connect all remaining $r-n$ dots to the rightmost of these $\alpha_s$ dots to get one big part of size $r-n+1$. In the top row, we add the unlabelled parts of $v$ in the same order as they appear in $v$. Then this new diagram $x$ lies in $U_v$. The construction of $x$ is illustrated by an example following the proof.
Define $$\rho_2(\epsilon\otimes (\boxtimes (\tau_i \otimes \sigma_i)) \otimes \eta) = \epsilon x \eta$$ for $\epsilon \in \Sigma_{l_1}$, $\tau_i \in \Sigma_{i\alpha_i}$, $\sigma_i \in \Sigma_{\alpha_i}$ and $\eta \in \Sigma_n$.
By the construction of $x$, $\epsilon x\eta = \tau d\sigma$, where as before $\tau = \epsilon\tau_1...\tau_s$, $d=d_v$ and $\sigma = \sigma_1...\sigma_s\eta$. 
We have to check that $\rho_2$ is well-defined, i.e. we have to check that whenever we move an element across a tensor, we get the same diagram in $U_v$. Let $\epsilon \in \Sigma_{l_1}$, $\xi =(\xi_1,...\xi_s) \in \Sigma_{(\alpha_1,2\alpha_2,...)}$, $\tau_i \in \Sigma_{i\alpha_i}$, $\zeta \in \Sigma_i \wr \Sigma_{\alpha_i}$, $\sigma_i \in \Sigma_{\alpha_i}$, $\vartheta = (\vartheta_1,,...,\vartheta_s) \in \Sigma_\alpha$ and $\eta \in \Sigma_n$. Denote the diagram $x$ by $x_{(\tau_i,\sigma_i)}$. 

It is $\rho_2(\epsilon\xi \otimes (\boxtimes (\tau_i \otimes \sigma_i))\otimes \eta) = \epsilon\xi x_{(\tau_i,\sigma_i)}\eta$ and $\rho_2(\epsilon \otimes (\boxtimes (\xi_i\tau_i \otimes \sigma_i)) \otimes \eta) = \epsilon x_{(\xi_i\tau_i,\sigma_i)} \eta$. There is no difference whether we first write the diagrams $x_i=\tau_id_i\sigma_i$ next to each other (and rearrange the bottom row, add dots and connections to the right) and then multiply by $\xi =(\xi_1,...\xi_s)$ or we start with multiplying each diagram $x_i$ by $\xi_i$ and then write them next to each other (and rearrange the bottom row, add dots and connections to the right). Hence, we can move $\xi=(\xi_1,...,\xi_s) \in \Sigma_{(\alpha_1,2\alpha_s,...)}$ across the first tensor product to or from the respective exterior tensor products.

It is $\rho_2(\epsilon\otimes(\boxtimes(\tau_i\zeta \otimes \sigma_i)) \otimes \eta) = \epsilon x_{(\tau_i\zeta,\sigma_i)}\eta$ and $\rho_2(\epsilon\otimes(\boxtimes(\tau_i \otimes \Pi(\zeta d_i)\sigma_i)) \otimes \eta) = \epsilon x_{(\tau_i,\Pi(\zeta d_i)\sigma_i)}\eta$. Since $\tau_i\zeta d_i\sigma_i = \tau_i d_i\Pi(\zeta d_i)\sigma_i$ for $\zeta \in \Sigma_i \wr \Sigma_{\alpha_i}$, we have $x_{(\tau_i\zeta,\sigma_i)} = x_{(\tau_i,\Pi(\zeta d_i)\sigma_i)}$ and we can move $\zeta$ across the interior tensor products within the exterior tensor product.

It is $\rho_2(\epsilon\otimes(\boxtimes(\tau_i \otimes \sigma_i\vartheta_i)) \otimes \eta) = \epsilon x_{(\tau_i,\sigma_i\vartheta_i)}\eta$ and $\rho_2(\epsilon\otimes(\boxtimes(\tau_i \otimes \sigma_i)) \otimes \vartheta\eta) = \epsilon x_{(\tau_i,\sigma_i)}\vartheta\eta$. Again, there is no difference whether we first write the diagrams $x_i$ next to each other (and rearrange the bottom row, add dots and connections to the right) and then multiply by $\vartheta$ from below, or we start with multiplying each diagram $x_i$ by $\vartheta_i$ and then write them next to each other (and rearrange the bottom row, add dots and connections to the right). Hence, we can move an element $\vartheta=(\vartheta_1,...,\vartheta_s) \in \Sigma_\alpha$ across the third interior tensor product.    

This shows that $\rho_2$ is well-defined. We see that $\rho_2$ is a $(k\Sigma_{l_1},k\Sigma_n)$-bimodule homomorphism analogously to Lemma \ref{lemma: psi}. 

We continue with an analysis of the image of $\rho_2$. Assume that $1_{\Sigma_l}$ is chosen as one of the coset representatives $\omega$ of $\Sigma_l/\Sigma_{(l_1,l_2)}$, cf. Subsection \ref{subsec: U as tensor product}, page \pageref{cosets omega}.
Using the maps $\tilde{\psi}$ from Lemma \ref{lemma: psi} and $\theta$ from Lemma \ref{lemma: theta}, we see that $\epsilon x \eta = \tau d \sigma$ is sent to $1_{\Sigma_l} \otimes ((\tau \otimes \sigma) \boxtimes (1_{\Sigma_{l_2}} \otimes 1)) \in k\Sigma_l \tensorover{k\Sigma_{(l_1,l_2)}} \left( (k\Sigma_{l_1} \tensorover{k\prod_\alpha} k\Sigma_n) \boxtimes (k\Sigma_{l_2} \tensorover{k\prod_\beta} k)\right)$ by $\theta \circ \tilde{\psi}$. So the image of $\rho_2$ is isomorphic to $k\Sigma_{l_1} \tensorover{k\prod_\alpha} k\Sigma_n$ and $\rho_3 \circ \theta \circ \tilde{\psi} \circ \rho_2$ is a candidate for the inverse of $\rho_1$, where $\rho_3$ is the projection $$k\Sigma_l \tensorover{k\Sigma_{(l_1,l_2)}} \left((k\Sigma_{l_1} \tensorover{k\prod_\alpha} k\Sigma_n) \boxtimes (k\Sigma_{l_2} \tensorover{k\prod_\beta} k)\right) \twoheadrightarrow k\Sigma_{l_1} \tensorover{k\prod_\alpha} k\Sigma_n.$$ Note that $\rho_3 \circ \theta$ is the identity on $k\Sigma_{l_1} \tensorover{k\prod_\alpha} k\Sigma_n \subset k\Sigma_l \tensorover{k(\prod_\alpha \times \prod_\beta)} k\Sigma_n$. 
  
Let $\tau \otimes \sigma \in k\Sigma_{l_1} \tensorover{k\prod_\alpha} k\Sigma_n$ as before. Then 

$\begin{aligned} 
(\rho_3 \circ \theta \circ \tilde{\psi} \circ \rho_2 \circ \rho_1)(\tau \otimes \sigma) & = (\rho_3 \circ \theta \circ \tilde{\psi} \circ \rho_2)(\epsilon \otimes (\boxtimes (\tau_i \otimes \sigma_i))\otimes \eta) \\ 
& = (\rho_3 \circ \theta \circ \tilde{\psi})(\epsilon x \eta) \\
& = (\rho_3 \circ \theta \circ \tilde{\psi})(\tau d \sigma) \\
&  = (\rho_3 \circ \theta)(\tau \otimes \sigma) \\
& = \tau \otimes \sigma \end{aligned}$

Now let $\epsilon \otimes (\boxtimes (\tau_i \otimes \sigma_i)) \otimes \eta \in k\Sigma_{l_1} \tensorover{k\Sigma_{(\alpha_1,2\alpha_2,...,s\alpha_s)}} \left(\boxtimes (k\Sigma_{i\alpha_i} \tensorover{k(\Sigma_i \wr \Sigma_{\alpha_i})} k\Sigma_{\alpha_i})\right) \tensorover{k\Sigma_\alpha} k\Sigma_n$, then 

$\begin{aligned} 
(\rho_1 \circ \rho_3 \circ \theta \circ \tilde{\psi} \circ \rho_2)(\epsilon \otimes (\boxtimes(\tau_i \otimes \sigma_i))\otimes \eta)) & = (\rho_1 \circ \rho_3 \circ \theta \circ \tilde{\psi})(\epsilon x \eta) \\
& = (\rho_1 \circ \rho_3 \circ \theta \circ \tilde{\psi})(\epsilon\tau_1...\tau_s d \sigma_1...\sigma_s\eta) \\
& = (\rho_1 \circ \rho_3 \circ \theta)(\epsilon\tau_1...\tau_s \otimes \sigma_1...\sigma_s\eta) \\
& = \rho_1(\epsilon\tau_1...\tau_s \otimes \sigma_1...\sigma_s\eta) \\
& = \epsilon \otimes (\boxtimes (\tau_i \otimes \sigma_i)) \otimes \eta) \end{aligned}$

This shows that $\rho_3 \circ \theta \circ \tilde{\psi} \circ \rho_2$ is indeed the two-sided inverse of $\rho_1$.  
\end{proof}

We illustrate the isomorphism with an example. 
\begin{ex}\label{ex sec generalised Foulkes} Let $l_1=12$, $l_2 = 3$, $n=5$, $\alpha_1=0$, $\alpha_2 = 3$, $\alpha_3=2$ and $\beta_1 = \beta_2 = 1$. Then $v = \xysmall{\circ \tra[r] & \circ & \circ \tra[r] & \circ & \circ \tra[r] & \circ & \circ \tra[r] & \circ \tra[r] & \circ & \circ \tra[r] & \circ \tra[r] & \circ & \bullet & \bullet \tra[r] & \bullet}$. Assume that we have fixed coset representatives of $\Sigma_{l_1}/\Sigma_{(6,6)}$ and $\Sigma_{(3,2)}\backslash \Sigma_n$ such that $\epsilon = (2748)$ and $\eta=(1234)$ are such representatives, respectively. Let $\tau = (1347)(285)(9 \, 11)$ and $\sigma =(134)$, then $\tau \in \epsilon\Sigma_{(6,6)}$ and $\sigma \in \Sigma_{(3,2)}\eta$. We have $$\begin{aligned}\tau d \sigma &= &\begin{minipage}{4cm}
\xymatrixrowsep{10pt}\xymatrixcolsep{10pt}
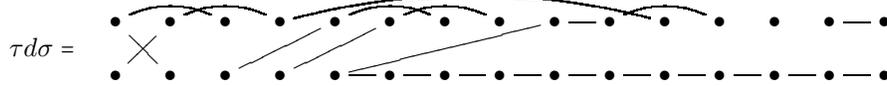
\xymatrix{
\bullet \arcu[rr]\tra[dr] & \bullet \arcu[rr]\tra[dl] & \bullet & \bullet\arcU[rrrrrrr] & \bullet\tra[dll]\arcu[rr] & \bullet \arcu[rr]\tra[dll] & \bullet & \bullet & \bullet\tra[r]\tra[dllll] & \bullet \arcu[rr] & \bullet & \bullet & \bullet & \bullet \tra[r] & \bullet \\ 
\bullet & \bullet& \bullet& \bullet& \bullet\tra[r]&\bullet\tra[r]&\bullet\tra[r]&\bullet\tra[r]&\bullet\tra[r]&\bullet\tra[r]&\bullet\tra[r]&\bullet \tra[r] & \bullet \tra[r] & \bullet \tra[r] & \bullet}
\end{minipage}\end{aligned}$$
Since $\tau = \epsilon (134)(25)(78)(9 \, 11)$ and $\sigma = (12)\eta$, the diagram $x$ is calculated as follows. $$x_2 = \tau_2d_2\sigma_2 = \begin{minipage}{4cm}
\xysmall{\bullet \arcu[rr]\tra[d] & \bullet \arcu[rrrr]\tra[dr] & \bullet & \bullet\tra[r]\tra[dll] & \bullet & \bullet \\
\bullet & \bullet & \bullet \tra[r] & \bullet \tra[r] & \bullet\tra[r] & \bullet}
\end{minipage}$$
and 
$$x_3 = \tau_3d_3\sigma_3 =  \begin{minipage}{4cm}
\xysmall{\bullet \tra[r]\tra[d] & \bullet \arcu[rrr] & \bullet \tra[r]\tra[dl]& \bullet\arcu[rr] & \bullet & \bullet \\
\bullet & \bullet \tra[r] & \bullet \tra[r] & \bullet \tra[r] & \bullet\tra[r] & \bullet}
\end{minipage}$$
with $\tau_2=(134)(25)$, $\tau_3 = (78)(9 \, 11)$, $\sigma_2=(12)$ and $\sigma_3=1_{\Sigma_2}$. Combining these two diagrams and rearranging the bottom row yields $$
x' = \begin{minipage}{4cm}
\xysmall{\bullet \arcu[rr]\tra[d] & \bullet \arcu[rrrr]\tra[dr] & \bullet & \bullet\tra[r]\tra[dll] & \bullet & \bullet & \bullet \tra[r]\tra[d] & \bullet \arcu[rrr] & \bullet \tra[r]\tra[dl]& \bullet\arcu[rr] & \bullet & \bullet \\
\bullet & \bullet & \bullet \tra[r] & \bullet \tra[r] & \bullet\tra[r] & \bullet & \bullet & \bullet \tra[r] & \bullet \tra[r] & \bullet \tra[r] & \bullet\tra[r] & \bullet}
\end{minipage} $$ and 
$$x = \begin{minipage}{4cm}
\xysmall{\bullet \arcu[rr]\tra[d] & \bullet \arcu[rrrr]\tra[dr] & \bullet & \bullet\tra[r]\tra[dll] & \bullet & \bullet & \bullet \tra[r]\tra[dlll] & \bullet \arcu[rrr] & \bullet \tra[r]\tra[dllll]& \bullet\arcu[rr] & \bullet & \bullet & \bullet & \bullet \tra[r] & \bullet \\
\bullet & \bullet & \bullet  & \bullet  & \bullet\tra[r] & \bullet \tra[r] & \bullet \tra[r] & \bullet \tra[r] & \bullet \tra[r] & \bullet \tra[r] & \bullet\tra[r] & \bullet \tra[r] & \bullet \tra[r] & \bullet \tra[r] & \bullet }
\end{minipage} $$
We then get $$\begin{aligned}
\epsilon x \eta & = & \begin{minipage}{5cm}
\xymatrixrowsep{10pt}\xymatrixcolsep{10pt}\xymatrix{
\bullet \arcu[rr]\tra[dr] & \bullet \arcu[rr]\tra[dl] & \bullet & \bullet \arcU[rrrrrrr] & \bullet\arcu[rr]\tra[dll] & \bullet\arcu[rr]\tra[dll] & \bullet & \bullet & \bullet\tra[r]\tra[dllll]& \bullet\arcu[rr]& \bullet& \bullet & \bullet & \bullet \tra[r] & \bullet\\
\bullet &\bullet&\bullet&\bullet&\bullet\tra[r]&\bullet\tra[r]&\bullet\tra[r]&\bullet\tra[r]&\bullet\tra[r]&\bullet\tra[r]&\bullet\tra[r]&\bullet\tra[r] & \bullet\tra[r] & \bullet\tra[r] & \bullet}
\end{minipage} & = \tau d \sigma. \end{aligned} $$
\end{ex}

The following Lemma \ref{lemma: permutation module} gives an interpretation for the smaller diagrams $\tau_id_i\sigma_i$ obtained in the above proof. It requires the definition of a right $k\Sigma_{\alpha_i}$-module structure on  $M^{(i^{\alpha_i})} = k\Sigma_{i\alpha_i} \tensorover{k\Sigma_{(i^{\alpha_i})}} k$ which is compatible with the usual left $k\Sigma_{i\alpha_i}$-module structure to obtain a bimodule. 

\begin{dfprop}\label{defprop} The permutation module $M^{(i^{\alpha_i})} = k\Sigma_{i\alpha_i} \tensorover{k\Sigma_{(i^{\alpha_i})}} k$ has a right $k\Sigma_{\alpha_i}$-module structure which is compatible with the usual left $k\Sigma_{i\alpha_i}$-module structure. 
\end{dfprop}

\begin{proof}
Let $\tau \in \Sigma_{i\alpha_i}$ and $\sigma \in \Sigma_{\alpha_i}$. Set $(\tau \otimes 1)\sigma = \tau \hat{\sigma} \otimes 1$ (where the tensor product is taken over $k\Sigma_{(i^{\alpha_i})}$) for some $\hat{\sigma} \in \Sigma_i \wr \Sigma_{\alpha_i}$ such that the image of $\hat{\sigma}$ under the canonical epimorphism $\Sigma_i \wr \Sigma_{\alpha_i} \twoheadrightarrow \Sigma_{\alpha_i}$ is $\sigma$, i.e. such that $\Pi(\hat{\sigma}d_i)=\sigma$. Extend this multiplication linearly.

First, we have to show that the given right module structure is independent of the choice of $\hat{\sigma}$. For each $\sigma \in \Sigma_{\alpha_i}$ we fix an element $\hat{\sigma} \in \Sigma_i \wr \Sigma_{\alpha_i}$ such that $\Pi(\hat{\sigma}d_i) = \sigma$. If $\bar{\sigma} \in \Sigma_i\wr\Sigma_{\alpha_i}$ is another element with $\Pi(\bar{\sigma}d_i) = \sigma$, then there must be an element $\zeta \in \Sigma_{(i^{\alpha_i})}$ such that $\bar{\sigma} = \hat{\sigma}\zeta$. 
Hence $\tau\bar{\sigma}\otimes 1 = \tau\hat{\sigma}\zeta\otimes 1 = \tau\hat{\sigma}\otimes 1$ and the module structure is well-defined. 

Now, let $\eta \in \Sigma_{i\alpha_i}$. Then $(\eta(\tau \otimes 1))\sigma = (\eta\tau \otimes 1)\sigma = \eta\tau\hat{\sigma} \otimes 1 = \eta((\tau \otimes 1)\sigma)$, so $M^{(i^{\alpha_i})}$ is in fact a $(k\Sigma_{i\alpha_i},k\Sigma_{\alpha_i})$-bimodule.  
\end{proof}

\begin{rk}
The right module structure is based on the $\ast$-action we encountered in Lemma \ref{lemma: Foulkes module}.
\end{rk}

\begin{lem}\label{lemma: permutation module}
There is an isomorphism $$k\Sigma_{i\alpha_i} \tensorover{k(\Sigma_i \wr \Sigma_{\alpha_i})} k\Sigma_{\alpha_i} \simeq k\Sigma_{i\alpha_i} \tensorover{k\Sigma_{(i^{\alpha_i})}} k$$ of $(k\Sigma_{i\alpha_i},k\Sigma_{\alpha_i})$-bimodules.
\end{lem}

\begin{proof}

We check that the map $$\begin{aligned} \varphi:\, & k\Sigma_{i\alpha_i} \tensorover{k(\Sigma_i \wr \Sigma_{\alpha_i})} k\Sigma_{\alpha_i} &\longrightarrow & \, k\Sigma_{i\alpha_i} \tensorover{k\Sigma_{(i^{\alpha_i})}} k \\& \tau \otimes \sigma &\longmapsto &\, \tau \hat{\sigma} \otimes 1  & \text{for } \hat{\sigma} \in \Sigma_i \wr \Sigma_{\alpha_i} \text{ with } \Pi(\hat{\sigma}d_i) =\sigma\end{aligned}$$ is well-defined. We have seen in Definition/Proposition \ref{defprop} that it is independent of the choice of $\hat{\sigma}$. Let $\Phi$ be the corresponding map $k\Sigma_{i\alpha_i} \times k\Sigma_{\alpha_i} \longrightarrow k\Sigma_{i\alpha_i} \tensorover{k\Sigma_{(i^{\alpha_i})}} k$ with $\Phi(\tau,\sigma)=\tau\hat{\sigma}\otimes 1$ and let $\xi \in \Sigma_i\wr\Sigma_{\alpha_i}$. Then $\Phi(\tau\xi,\sigma)=\tau\xi\hat{\sigma} \otimes 1$ and $\Phi(\tau, \Pi(\xi d_i)\sigma)=\tau (\widehat{\Pi(\xi d_i)\sigma}) \otimes 1$. Since $\hat{\sigma} \in \Sigma_i \wr \Sigma_{\alpha_i}$, we have $$\hat{\sigma}d_i= d_i\sigma$$ which can be easily seen by comparison of the top and bottom rows of the diagrams (in both cases, the permutation induced by propagating lines is $\sigma$). Hence $\Pi(\xi \hat{\sigma}d_i) = \Pi(\xi d_i \sigma) = \Pi(\xi d_i)\sigma$. It follows that $\Phi(\tau\xi,\sigma)=\tau\xi\hat{\sigma}\otimes 1=\tau (\widehat{\Pi(\xi d_i)\sigma}) \otimes 1 =\Phi(\tau,\Pi(\xi d_i)\sigma)$, so $\varphi$ is well-defined.

Let $\pi, \tau \in \Sigma_{i\alpha_i}$ and $\sigma, \eta \in \Sigma_{\alpha_i}$. Then $\pi \varphi(\tau \otimes \sigma) \eta = \pi(\tau \hat{\sigma} \otimes 1)\eta = \pi \tau \hat{\sigma}\hat{\eta} \otimes 1$. Since $\Pi(\hat{\sigma}\hat{\eta}d_i) = \Pi(\hat{\sigma}d_i \eta) = \Pi(\hat{\sigma}d_i)\eta =  \sigma\eta$, we have $\pi \tau \hat{\sigma}\hat{\eta} \otimes 1 = \pi \tau (\widehat{\sigma\eta}) \otimes 1 = \varphi(\pi\tau \otimes \sigma\eta)$, so $\varphi$ is a homomorphism of bimodules. 

Since for any element $\sigma \in \Sigma_{\alpha_i}$ there is an element $\hat{\sigma} \in \Sigma_i \wr \Sigma_{\alpha_i}$ such that $\Pi(\hat{\sigma}d_i)=\sigma$, we can write any element $\tau \otimes \sigma \in \Sigma_{i\alpha_i} \tensorover{\Sigma_i \wr \Sigma_{\alpha_i}} \Sigma_{\alpha_i}$ as $\tau \hat{\sigma} \otimes 1$. This shows that the inverse of $\varphi$ is given by $$\begin{aligned}
\varphi^{-1}:\, & k\Sigma_{i\alpha_i} \tensorover{k\Sigma_{(i^{\alpha_i})}} k & \longrightarrow &\, k\Sigma_{i\alpha_i} \tensorover{k(\Sigma_i \wr \Sigma_{\alpha_i})} k\Sigma_{\alpha_i} \\ &\tau \otimes 1 & \longmapsto &\, \tau \otimes 1
\end{aligned}$$ 
\end{proof}

\begin{cor}\label{cor: prop labellled}
The $(k\Sigma_{l_1},k\Sigma_n)$-bimodule $k\Sigma_{l_1} \tensorover{k\prod_\alpha} k\Sigma_n$ is isomorphic to $$k\Sigma_{l_1} 
\tensorover{k\Sigma_{(\alpha_1,2\alpha_2,...,s\alpha_s)}} \left( (k\Sigma_{\alpha_1} \tensorover{k\Sigma_{(1^{\alpha_1})}} k) \boxtimes ... \boxtimes (k\Sigma_{s\alpha_s} \tensorover{k\Sigma_{(s^{\alpha_s})}} k)\right) \tensorover{k\Sigma_\alpha} k\Sigma_n.$$
\end{cor}

\begin{proof}
Use Proposition \ref{prop: decomposition of labelled parts into small diagrams} and Lemma \ref{lemma: permutation module}.
\end{proof}

We return to the module $k\Sigma_{l_1} \tensorover{k\prod_\alpha} k\Sigma_n$, now regarded only as left $k\Sigma_{l_1}$-module instead of a bimodule. 

\begin{prop}\label{prop: labelled left}
As left $k\Sigma_{l_1}$-module, $k\Sigma_{l_1} \tensorover{k\prod_\alpha} k\Sigma_n$ is isomorphic to a direct sum of copies of $k\Sigma_{l_1} \tensorover{k\Sigma_{(\alpha_1,2\alpha_2,...,s\alpha_s)}} \left((k\Sigma_{\alpha_1} \tensorover{k\Sigma_{(1^{\alpha_1})}} k) \boxtimes ... \boxtimes (k\Sigma_{s\alpha_s} \tensorover{k\Sigma_{(s^{\alpha_s})}} k)\right)$.
\end{prop}

\begin{proof}
By Corollary \ref{cor: prop labellled}, we have the bimodule-isomorphism $$k\Sigma_{l_1} \tensorover{k\prod_\alpha} k\Sigma_n \simeq k\Sigma_{l_1} \tensorover{k\Sigma_{(\alpha_1,2\alpha_2,...,s\alpha_s)}} \left((k\Sigma_{\alpha_1} \tensorover{k\Sigma_{(1^{\alpha_1})}} k) \boxtimes ... \boxtimes (k\Sigma_{s\alpha_s} \tensorover{k\Sigma_{(s^{\alpha_s})}} k)\right) \tensorover{k\Sigma_\alpha} k\Sigma_n.$$ 
If $w \in V_n^{l_1}/_\sim$ has no unlabelled dot, then $U_w \simeq k\Sigma_{l_1} \tensorover{k\prod_\alpha} k\Sigma_n$ by Lemma \ref{lemma: psi}, so $k\Sigma_{l_1} \tensorover{k\prod_\alpha} k\Sigma_n$ is isomorphic to the $(k\Sigma_{l_1},k\Sigma_n)$-bimodule generated by the partition diagram $d_w$ whose top row parts are all connected to a single dot in the bottom row (with one exception, which is connected to a bottom row part of size $r-n+1$, but this is regarded as a single dot as well by the action of $k\Sigma_n$). The action of $k\Sigma_n$ permutes the propagating lines attached to parts of different sizes, while it leaves the top row unchanged. This cannot be obtained by the left $k\Sigma_{l_1}$-action. Hence, $k\Sigma_{l_1} \tensorover{k\prod_\alpha} k\Sigma_n$ decomposes into a direct sum of copies of $$k\Sigma_{l_1} \tensorover{k\Sigma_{(\alpha_1,2\alpha_2,...,s\alpha_s)}} \left((k\Sigma_{\alpha_1} \tensorover{k\Sigma_{(1^{\alpha_1})}} k) \boxtimes ... \boxtimes (k\Sigma_{s\alpha_s} \tensorover{k\Sigma_{(s^{\alpha_s})}} k)\right),$$ one for each configuration of propagating lines which cannot be obtained from the generator by left multiplication, i.e. one for each coset in $\Sigma_\alpha\backslash \Sigma_n$. 
\end{proof}

Corollary \ref{cor: prop labellled} and Proposition \ref{prop: labelled left} prove part (\ref{TMR labelled}) of Theorem \ref{TMR}. 

\begin{cor} \label{cor: labelled specht} The left $k\Sigma_{l_1}$-module
$k\Sigma_{l_1} \tensorover{k\prod_\alpha} k\Sigma_n$ has a filtration by dual Specht modules.
\end{cor}

\begin{proof} It is well known that the Young permutation modules $k\Sigma_{i\alpha_i} \tensorover{k\Sigma_{(i^{\alpha_i})}} k$ have filtrations by dual $k\Sigma_{i\alpha_i}$-Specht modules. The characteristic-free version of the Littlewood-Richardson rule \cite{JP} then shows that the induced module $$k\Sigma_{l_1} \tensorover{k\Sigma_{(\alpha_1, 2\alpha_2, ..., s\alpha_s)}} \left((k\Sigma_{\alpha_1} \tensorover{k\Sigma_{(1^{\alpha_1})}} k) \boxtimes ... \boxtimes (k\Sigma_{s\alpha_s} \tensorover{k\Sigma_{(s^{\alpha_s})}} k )\right)$$ has a dual Specht filtration as well. 
\end{proof}

\begin{prop}\label{prop: right permutation module free} 
$k\Sigma_{i\alpha_i} \tensorover{k\Sigma_{(i^{\alpha_i})}} k$ is free over $k\Sigma_{\alpha_i}$.
\end{prop}

\begin{proof}
The permutation module $k\Sigma_{i\alpha_i} \tensorover{k\Sigma_{(i^{\alpha_i})}} k$ has a basis given by $(i^{\alpha_i})$-tabloids. The right action of $k\Sigma_{\alpha_i}$ permutes the $\alpha_i$ rows of these tabloids. Let $\bar{t}$ be any $(i^{\alpha_i})$-tabloid, then the orbit of $\bar{t}$ under this action is isomorphic to $\Sigma_{\alpha_i}$. In particular, $k\Sigma_{i\alpha_i} \tensorover{k\Sigma_{(i^{\alpha_i})}} k$ is a direct sum of copies of $k\Sigma_{\alpha_i}$, one for each row configuration (in this context, a row configuration fixes which entries belong to the same row, but not which row they belong to). 
\end{proof}

\subsection{Specht filtration of $\boldsymbol{e_l(A/J_{n-1})e_n}$}\label{subsec: eAe Specht}

\begin{prop}[Theorem \ref{TMR}, part (\ref{TMR Specht})]\label{prop: eAe Specht filtered}
Let $\text{\textnormal{char}} k = 0$ or $\text{\textnormal{char}} k > \lfloor \frac{l-n}{3} \rfloor$. Then $e_l(A/J_{n-1})e_n \in \mathcal{F}_l(S)$.
\end{prop}

\begin{proof}
By Lemmas \ref{lemma: psi} and \ref{lemma: theta}, the summands of $e_l(A/J_{n-1})e_n$ are of the form $$k\Sigma_l \tensorover{k\Sigma_{(l_1,l_2)}} ((k\Sigma_{l_1} \tensorover{k\prod_\alpha} k\Sigma_n) \boxtimes (k\Sigma_{l_2} \tensorover{k\prod_\beta} k)),$$ where $\prod_\alpha \simeq \prod\limits_i (\Sigma_i\wr\Sigma_{\alpha_i})$ and $\prod_\beta \simeq \prod\limits_i (\Sigma_i\wr\Sigma_{\beta_i})$, $l_1=\sum\limits_{i=1}^s i\alpha_i$ and $l_2=\sum\limits_{i=1}^t i\beta_i$. We know that $k\Sigma_{l_1} \tensorover{k\prod_\alpha} k\Sigma_n \in \mathcal{F}_{l_1}(S)$ by Corollary \ref{cor: labelled specht} and $k\Sigma_{l_2} \tensorover{k\prod_\beta} k \in \mathcal{F}_{l_2}(S)$ provided $\text{\textnormal{char}}k = 0$ or $\text{\textnormal{char}}k > \max \beta_i$ by Corollary \ref{cor: iterated Foulkes}. We then apply the characteristic-free version of the Littlewood-Richardson rule \cite{JP} to see that $$k\Sigma_l \tensorover{k\Sigma_{(l_1,l_2)}} ((k\Sigma_{l_1} \tensorover{k\prod_\alpha} k\Sigma_n) \boxtimes (k\Sigma_{l_2} \tensorover{k\prod_\beta} k))$$ lies in $\mathcal{F}_l(S)$. 

Since we want to examine the whole bimodule $e_l(A/J_{n-1})e_n$ and not just one of its summands $U_v$, we have to consider all top row configurations in $V_n^l/_\sim$ simultaneously.

The factor $k\Sigma_{2\beta_2} \tensorover{k(\Sigma_2 \wr \Sigma_{\beta_2})} k$ is the stabilizer of unlabelled parts of size $2$ and it is dual Specht filtered by \cite[Proposition 8]{P} in any characteristic. For $i>2$, the factor  $k\Sigma_{i\beta_i} \tensorover{k(\Sigma_i \wr \Sigma_{\beta_i})} k$ is dual Specht filtered if $\text{\textnormal{char}} k = 0$ or $\text{\textnormal{char}} k > \beta_i$ by Corollary \ref{cor: Foulkes Specht}. The maximal amount of unlabelled parts of a certain size greater than two occurs in the summands $U_v$ where $v$ consists of $n$ labelled singletons and $\lfloor \frac{l-n}{3} \rfloor$ unlabelled parts of size $3$. The remaining $0,1$ or $2$ dots form additional unlabelled parts. Hence, we have to assume $\text{\textnormal{char}}k > \lfloor \frac{l-n}{3} \rfloor$. This lower bound for the characteristic of the field comes from Lemma \ref{lemma: Foulkes module}. It is sufficient, but probably too strong.  
\end{proof}

This concludes the proof of Theorem \ref{TMR}.

\section{Restriction of Cell Modules}\label{sec: result}
We are now able to put the results of the previous section together to obtain criteria for the restriction of a cell module of $P_k(r,\delta)$ to a group algebra of a symmetric group with index $l\leq r$ to admit a dual Specht filtration. This last section is dedicated to the proof of Theorem \ref{main theorem}.

\begin{thm1}
Let $A$ be the partition algebra $P_k(r,\delta)$ with $\delta \neq 0$. Let $0\leq n \leq l \leq r$. Let $\nu$ be a partition of $n$. 
The restriction $e_lAe_n \tensorover{e_nAe_n}S_\nu$ of the cell module $Ae_n\tensorover{e_nAe_n}S_\nu$ to $k\Sigma_l -\mod$ admits a dual Specht filtration if the following holds
\begin{enumerate}
\item $\text{\textnormal{char}} k = 0$ or $\text{\textnormal{char}} k > \lfloor \frac{l-n}{3} \rfloor$ \end{enumerate}
 and
\begin{enumerate}\setcounter{enumi}{1}
\item for all $i \in \{1,...,l\}, \alpha_i \in \{1,...,\lfloor \frac{l}{i}\rfloor\}$ and for all partitions $\lambda^i$ of $\alpha_i$, the generalised Foulkes modules with inner twists $$k\Sigma_{i\alpha_i} \tensorover{k(\Sigma_i \wr \Sigma_{\alpha_i})} S_{\lambda^i} \simeq (k\Sigma_{i\alpha_i} \tensorover{k\Sigma_{(i^{\alpha_i})}} k) \tensorover{k\Sigma_{\alpha_i}} S_{\lambda^i}$$ admit filtrations by dual $k\Sigma_{i\alpha_i}$-Specht modules.
\end{enumerate}
\end{thm1}

We prove this in Subsection \ref{subsec: proof positive} and conjecture that the theorem fails in full generality in Subsection \ref{subsec: proof negative}. In Subsection \ref{subsec: necessity of conditions}, we discuss the assumptions needed for the theorem. 

Our main result is about the restriction $e_l(A/J_{n-1})e_n \tensorover{k\Sigma_n} S_\nu$ of a cell module, where $\nu$ is a partition of $n$, and in the proof of the theorem we will encounter a tensor factor of the form $k\Sigma_{i\alpha_i} \tensorover{k(\Sigma_i \wr\Sigma_{\alpha_i})} S_{\lambda^i}$ where $\lambda^i$ is a partition of $\alpha_i$. This motivates the following lemma about the generalised Foulkes module with inner twists $k\Sigma_{i\alpha_i} \tensorover{k(\Sigma_i \wr \Sigma_{\alpha_i})}  (S_{(i)} \oslash S_{\lambda^i})$, where the product  $X \oslash Y$ of $X \in k\Sigma_i-\mod$ and $Y \in k\Sigma_{\alpha_i}-\mod$ is defined as follows\footnote{see, for example \cite{PagetWildonGeneralisedFoulkes}.}. Let $X^{\boxtimes \alpha_i} = X \boxtimes ... \boxtimes X$ be the $\alpha_i$-fold exterior tensor product of $X$ with itself. Then $\Sigma_{\alpha_i}$ acts on $X^{\boxtimes \alpha_i}$ by place permutation, turning $X^{\boxtimes \alpha_i}$ into a $k(\Sigma_i \wr \Sigma_{\alpha_i})$-module with action $$(\xi_1,...,\xi_{\alpha_i};\zeta)\cdot (x_1\boxtimes ... \boxtimes x_{\alpha_i}):=\xi_1 x_{\zeta^{-1}(1)} \boxtimes ... \boxtimes \xi_{\alpha_i} x_{\zeta^{-1}(\alpha_i)}$$ where $\xi_j \in \Sigma_i$, $x_j \in X$ for $j = 1,...,\alpha_i$ and $\zeta \in \Sigma_{\alpha_i}$. 
Since $k\Sigma_{\alpha_i}$ is a quotient of $k(\Sigma_i \wr \Sigma_{\alpha_i})$, the $k\Sigma_{\alpha_i}$-module $Y$ is also a $k(\Sigma_i \wr \Sigma_{\alpha_i})$-module, which we denote by $\textnormal{Inf}_{\Sigma_{\alpha_i}}^{\Sigma_i \wr \Sigma_{\alpha_i}} Y \simeq k\Sigma_{\alpha_i} \tensorover{k\Sigma_{\alpha_i}} Y$.
Then $X \oslash Y$ is the $k(\Sigma_i \wr \Sigma_{\alpha_i})$-module $X^{\boxtimes \alpha_i} \boxtimes \textnormal{Inf}_{\Sigma_{\alpha_i}}^{\Sigma_i \wr \Sigma_{\alpha_i}} Y$.

\begin{lem} \label{lemma: generalised Foulkes}
The module $k\Sigma_{i\alpha_i} \tensorover{k(\Sigma_i \wr \Sigma_{\alpha_i})} k\Sigma_{\alpha_i} \tensorover{k\Sigma_{\alpha_i}} S_{\lambda^i}$ 
is isomorphic to the \emph{generalised Foulkes module with inner twists} $H_{\lambda^i}^{(i)} = k\Sigma_{i\alpha_i} \tensorover{k(\Sigma_i \wr \Sigma_{\alpha_i})} (S_{(i)} \oslash S_{\lambda^i})$. 
\end{lem}

\begin{rk} In the literature, a \emph{generalised Foulkes module} is defined using Specht modules instead of dual Specht modules. Since dual Specht modules are Specht modules tensored with the sign-module $S^{(1^m)}$, we call the modules in Lemma \ref{lemma: generalised Foulkes} generalised Foulkes modules \emph{with inner twists}.
\end{rk}

\begin{proof}
Applying the definition of the $\oslash$-product to our special case, we get $$S_{(i)} \oslash S_{\lambda^i} = (S_{(i)})^{\boxtimes \alpha_i} \boxtimes \textnormal{Inf}_{\Sigma_{\alpha_i}}^{\Sigma_i \wr \Sigma_{\alpha_i}} S_{\lambda^i} \simeq k^{\boxtimes \alpha_i} \boxtimes (k\Sigma_{\alpha_i} \tensorover{k\Sigma_{\alpha_i}} S_{\lambda^i}).$$ Hence $H_{\lambda^i}^{(i)} \simeq k\Sigma_{i\alpha_i} \tensorover{k(\Sigma_i \wr \Sigma_{\alpha_i})} (k^{\boxtimes \alpha_i} \boxtimes (k\Sigma_{\alpha_i} \tensorover{k\Sigma_{\alpha_i}} S_{\lambda^i})) \simeq k\Sigma_{i\alpha_i} \tensorover{k(\Sigma_i \wr \Sigma_{\alpha_i})} k\Sigma_{\alpha_i} \tensorover{k\Sigma_{\alpha_i}} S_{\lambda^i}$, where the left $k(\Sigma_i \wr \Sigma_{\alpha_i})$-module structure of $k\Sigma_{\alpha_i}$ is induced by the canonical epimorphism $\Sigma_i \wr\Sigma_{\alpha_i} \twoheadrightarrow \Sigma_{\alpha_i}$, $(\xi_1,...,\xi_{\alpha_i};\zeta) \mapsto \zeta$. 
\end{proof}

\subsection{Proof and corollary of Theorem \ref{main theorem}}\label{subsec: proof positive}

As we have seen in Section \ref{sec: restriction cell modules}, the $(k\Sigma_l,k\Sigma_n)$-bimodule $e_l(A/J_{n-1})e_n$ decomposes into $\bigoplus\limits_{v \in V_n^l/_\sim} U_v$, so the restriction $e_l(A/J_{n-1})e_n \tensorover{k\Sigma_n} S_\nu$ of a cell module decomposes into $\bigoplus\limits_{v \in V_n^l/_\sim} (U_v \tensorover{k\Sigma_n} S_\nu)$. Fix $v \in V_n^l/_\sim$. By Lemmas \ref{lemma: psi} and \ref{lemma: theta}, we have $$\begin{aligned}
U_v \tensorover{k\Sigma_n} S_\nu &\simeq k\Sigma_l \tensorover{k\Sigma_{(l_1,l_2)}} \left( (k\Sigma_{l_1} \tensorover{k\prod_\alpha} k\Sigma_n) \boxtimes (k\Sigma_{l_2} \tensorover{k\prod_\beta} k)\right) \tensorover{k\Sigma_n} S_\nu  \\
& \simeq k\Sigma_l \tensorover{k\Sigma_{(l_1,l_2)}} \left( (k\Sigma_{l_1} \tensorover{k\prod_\alpha} k\Sigma_n \tensorover{k\Sigma_n} S_\nu ) \boxtimes (k\Sigma_{l_2} \tensorover{k\prod_\beta} k)\right). 
\end{aligned}$$
By Corollary \ref{cor: iterated Foulkes}, the tensor factor $k\Sigma_{l_2} \tensorover{k\prod_\beta} k$ has a dual $k\Sigma_{l_2}$-Specht filtration, provided the characteristic of the field is large enough.

It remains to check that the tensor factor $k\Sigma_{l_1} \tensorover{k\prod_\alpha} k\Sigma_n \tensorover{k\Sigma_n} S_\nu$ admits a filtration by dual $k\Sigma_{l_2}$-Specht modules. By Proposition \ref{prop: decomposition of labelled parts into small diagrams}, $k\Sigma_{l_1} \tensorover{k\prod_\alpha} k\Sigma_n \tensorover{k\Sigma_n} S_\nu$ is isomorphic to $$\begin{aligned}  k\Sigma_{l_1} 
\tensorover{k\Sigma_{(\alpha_1,2\alpha_2,...,s\alpha_s)}} \left( \overset{s}{\underset{i=1}{\boxtimes}}(k\Sigma_{i\alpha_i} \tensorover{k(\Sigma_i \wr \Sigma_{\alpha_i})} k\Sigma_{\alpha_i}) \right) \tensorover{k\Sigma_\alpha} k\Sigma_n \tensorover{k\Sigma_n} S_\nu \\
\simeq k\Sigma_{l_1} \tensorover{k\Sigma_{(\alpha_1,2\alpha_2,...,s\alpha_s)}} \left( \overset{s}{\underset{i=1}{\boxtimes}}(k\Sigma_{i\alpha_i} \tensorover{k(\Sigma_i \wr \Sigma_{\alpha_i})} k\Sigma_{\alpha_i}) \right) \tensorover{k\Sigma_{\alpha}} S_\nu , \end{aligned}$$ so the action of $k\Sigma_n$ on $S_\nu$ is restricted to $k\Sigma_\alpha=k(\Sigma_{\alpha_1} \times ... \times \Sigma_{\alpha_s})$. Denote this restriction by $\leftidx{^{n}_{\alpha}}{\hspace{-3pt}\downarrow}{}\hspace{-2.5pt} S_\nu$. Then $\leftidx{^{n}_{\alpha}}{\hspace{-3pt}\downarrow}{}\hspace{-2.5pt} S_\nu$ is filtered by modules of the form $S_{\lambda^1} \boxtimes ... \boxtimes S_{\lambda^s}$, where each $S_{\lambda^i}$ is a dual Specht module in $k\Sigma_{\alpha_i}-\mod$. In particular, $$k\Sigma_{l_1} \tensorover{k\Sigma_{(\alpha_1,2\alpha_2,...,s\alpha_s)}} \left( \overset{s}{\underset{i=1}{\boxtimes}}(k\Sigma_{i\alpha_i} \tensorover{k(\Sigma_i \wr \Sigma_{\alpha_i})} k\Sigma_{\alpha_i}) \right) \tensorover{k\Sigma_{\alpha}} S_\nu$$ is filtered by modules of the form $$k\Sigma_{l_1} \tensorover{k\Sigma_{(\alpha_1,2\alpha_2,...,s\alpha_s)}} \left( \overset{s}{\underset{i=1}{\boxtimes}}(k\Sigma_{i\alpha_i} \tensorover{k(\Sigma_i \wr \Sigma_{\alpha_i})} k\Sigma_{\alpha_i}) \right) \tensorover{k\Sigma_{\alpha}} (\overset{s}{\underset{i=1}{\boxtimes}} S_{\lambda^i}).$$
By the action of $\Sigma_{\alpha}$ on the exterior tensor products, this is isomorphic to 
\begin{align*}
& k\Sigma_{l_1} \tensorover{k\Sigma_{(\alpha_1,2\alpha_2,...,s\alpha_s)}} \left( \overset{s}{\underset{i=1}{\boxtimes}}(k\Sigma_{i\alpha_i} \tensorover{k(\Sigma_i \wr \Sigma_{\alpha_i})} k\Sigma_{\alpha_i} \tensorover{k\Sigma_{\alpha_i}} S_{\lambda^i}) \right) \\
\simeq & \, k\Sigma_{l_1} \tensorover{k\Sigma_{(\alpha_1,2\alpha_2,...,s\alpha_s)}} \left( \overset{s}{\underset{i=1}{\boxtimes}}(k\Sigma_{i\alpha_i} \tensorover{k(\Sigma_i \wr \Sigma_{\alpha_i})} S_{\lambda^i}) \right). 
\end{align*}
By assumption \ref{main thm assumption gFm} of Theorem \ref{main theorem}, we have $k\Sigma_{i\alpha_i} \tensorover{k(\Sigma_i \wr \Sigma_{\alpha_i})} S_{\lambda^i} \in \mathcal{F}_{i\alpha_i}(S)$ for all $i = 1,...,s$ and it follows from the characteristic-free version of the Littlewood-Richardson rule \cite{JP} that 
$$k\Sigma_{l_1} \tensorover{k\Sigma_{(\alpha_1,2\alpha_2,...,s\alpha_s)}} \left( \overset{s}{\underset{i=1}{\boxtimes}}(k\Sigma_{i\alpha_i} \tensorover{k(\Sigma_i \wr \Sigma_{\alpha_i})}  S_{\lambda^i}) \right) \in \mathcal{F}_{l_1}(S).$$
This shows that $k\Sigma_{l_1}\tensorover{k\prod_\alpha} k\Sigma_n \tensorover{k\Sigma_n} S_\nu$ admits a filtration by modules which lie in $\mathcal{F}_{l_1}(S)$. This filtration can then be refined to get a filtration by dual Specht modules and thus $k\Sigma_{l_1}\tensorover{k\prod_\alpha} k\Sigma_n \tensorover{k\Sigma_n} S_\nu \in \mathcal{F}_{l_1}(S)$.  Another application of the Littlewood-Richardson rule shows that $U_v \tensorover{k\Sigma_n} S_\nu \in \mathcal{F}_l(S)$.

Considering all summands $U_v$, $v \in V_n^l/_\sim$, simultaneously, we get the assumption $\text{\textnormal{char}}k > \lfloor \frac{l-n}{3} \rfloor$ (or $\text{\textnormal{char}}k = 0$), as in the proof of Proposition \ref{prop: eAe Specht filtered}.

This concludes the proof of Theorem \ref{main theorem}. \qed\\

\begin{lem}[Wildon\footnote{Lemma \ref{lemma: Wildon} is based on an unpublished result by Mark Wildon. I am grateful to Mark Wildon for informing me about this result and its proof, and for giving me permission to include it here.}]\label{lemma: Wildon}
If both $i$ and $\alpha_i$ are strictly less than $\text{\textnormal{char}}k = p$ and $\lambda^i$ is a partition of $\alpha_i$
then the generalised Foulkes module with inner twists $H^{(i)}_{\lambda^i}$ has a filtration by dual Specht modules.
\end{lem}

\begin{proof}
Let $p=\text{\textnormal{char}}k > i$ and $p > \alpha_i$ and let $\lambda^i$ be a partition of $\alpha_i$. Then the Specht module $S^{\lambda^i}$ is a direct summand of the self-dual permutation module $M^{\lambda^i}$, so $S_{\lambda^i}$ is a direct summand of  $M^{\lambda^i}$. It follows that $H^{(i)}_{\lambda^i}= k\Sigma_{i\alpha_i} \tensorover{k(\Sigma_i \wr \Sigma_{\alpha_i})} (S_{(i)} \oslash S_{\lambda^i})$ is a direct summand of $k\Sigma_{i\alpha_i} \tensorover{k(\Sigma_i \wr \Sigma_{\alpha_i})} (M^{(i)} \oslash M^{\lambda^i})=:\textbf{M}^{(i)}_{\lambda^i}$.
By definition, this is isomorphic to $k\Sigma_{i\alpha_i} \tensorover{k(\Sigma_i \wr \Sigma_{\alpha_i})} \left( \underbrace{(M^{(i)} \boxtimes ... \boxtimes M^{(i)})}_{\alpha_i \text{ times}} \boxtimes (k\Sigma_{\alpha_i} \tensorover{k\Sigma_{\alpha_i}} M^{\lambda^i})\right)$. Now $M^{(i)}=k\Sigma_i \tensorover{k\Sigma_i} k \simeq k$ and $M^{\lambda^i} = k\Sigma_{\alpha_i} \tensorover{k\Sigma_{\lambda^i}} k$, so  $\textbf{M}^{(i)}_{\lambda^i} \simeq k\Sigma_{i\alpha_i} \tensorover{k(\Sigma_i \wr \Sigma_{\alpha_i})}\left((\underbrace{k \boxtimes ... \boxtimes k}_{\alpha_i \text{ times}}) \boxtimes (k\Sigma_{\alpha_i} \tensorover{k\Sigma_{\alpha_i}} k\Sigma_{\alpha_i} \tensorover{k\Sigma_{\lambda^i}} k)\right) \simeq k\Sigma_{i\alpha_i} \tensorover{k(\Sigma_i \wr \Sigma_{\alpha_i})}\left((\underbrace{k \boxtimes ... \boxtimes k}_{\alpha_i \text{ times}}) \boxtimes (k\Sigma_{\alpha_i} \tensorover{k\Sigma_{\lambda^i}} k)\right).$ The wreath product $\Sigma_i \wr \Sigma_{\alpha_i} \subseteq \Sigma_{i\alpha_i}$ acts on $k\boxtimes ... \boxtimes k$ via place permutation, but since all tensor factors are $k$ we can move everything across the tensor signs and ignore the place permutation. Then $\Sigma_i \wr \Sigma_{\alpha_i} \subseteq \Sigma_{i\alpha_i}$ acts trivially on $k\boxtimes ... \boxtimes k$ and via the canonical epimorphism $\Sigma_i \wr \Sigma_{\alpha_i} \twoheadrightarrow \Sigma_{\alpha_i}$ on $k\Sigma_{\alpha_i}\tensorover{k\Sigma_{\lambda^i}} k$. It follows that $\textbf{M}^{(i)}_{\lambda^i} \simeq k\Sigma_{i\alpha_i} \tensorover{k(\Sigma_i \wr \Sigma_{\alpha_i})}  k\Sigma_{\alpha_i} \tensorover{k\Sigma_{\lambda^i}} k$, where $k(\Sigma_i\wr \Sigma_{\alpha_i})$ acts on $k\Sigma_{\alpha_i}$ via the canonical epimorphism.
 
 This is a permutation module induced (in multiple steps) from subgroups without any $p$-elements, so $H^{(i)}_{\lambda^i}$ is projective. Therefore, $H^{(i)}_{\lambda^i}$ has a dual Specht filtration. 
\end{proof}

With the help of Lemma \ref{lemma: Wildon}, we can reformulate assumption (\ref{main thm assumption gFm}) of Theorem \ref{main theorem} in terms of the characteristic of the field. If $\alpha_i = 1$ then $H^{(i)}_{\lambda^i} \simeq k \in \mathcal{F}_{i\alpha_i}(S)$. The maximal size $i$ such that $\alpha_i > 1$ is $i= \lfloor \frac{l}{2} \rfloor$ and the maximal value of the integers $\alpha_i$ is  $\lfloor \frac{l}{2} \rfloor$ (for $i=2$). 
Hence we can simplify assumption (\ref{main thm assumption gFm}) and demand that $\text{\textnormal{char}}k > \lfloor\frac{l}{2}\rfloor$. This assumption might be stronger than the original assumption, but less complicated, so it is worth stating the following version of Theorem \ref{main theorem} as a corollary.

\begin{cor}\label{cor: thm 1}
Let $A$ be the partition algebra $P_k(r,\delta)$, $\delta \neq 0$ and let $0\leq n \leq l \leq r$. Let $\nu$ be a partition of $n$. 
The restriction $e_lAe_n \tensorover{e_nAe_n}S_\nu$ of the cell module $Ae_n\tensorover{e_nAe_n}S_\nu$ to $k\Sigma_l -\mod$ admits a dual Specht filtration if $\text{\textnormal{char}} k = 0$ or $\text{\textnormal{char}} k > \lfloor \frac{l}{2} \rfloor$.
\end{cor}

\subsection{Non-existence of dual Specht filtrations}\label{subsec: proof negative} 
In this subsection, we discuss why the theorem may not hold in full generality. 

\begin{lem}\label{lemma: rcm in F implies gFm in F} 
Let $\text{\textnormal{char}}k \geq 5$ and let $0 \leq n \leq l \leq r$ with $l=mn$. If $e_lAe_n \tensorover{e_nAe_n} S_\nu$ admits a filtration by dual Specht modules, then so does $H^{(m)}_\nu$. 
\end{lem}

\begin{proof} Let $e_lAe_n \tensorover{e_nAe_n} S_\nu \in \mathcal{F}_l(S)$. If $\text{\textnormal{char}}k \geq 5$, it follows from \cite[Theorem 3.6.1]{HN} that the category $\mathcal{F}_l(S)$ is closed under taking direct summands. So for each top row configuration $v \in V_n^l/_\sim$ we must have $U_v \tensorover{k\Sigma_n} S_\nu \in \mathcal{F}_l(S)$. Since $l=mn$, one summand is indexed by $$w:= \xymatrixcolsep{5pt}\xymatrix{\circ \tra[r] &  ... \tra[r] & \circ & \circ \tra[r] &  ... \tra[r] & \circ & ... & \circ \tra[r] &  ... \tra[r] & \circ} \in V_n^l,$$ the partial diagram consisting of exactly $n$ labelled parts of size $m$. From Lemma \ref{lemma: psi}, Proposition \ref{prop: decomposition of labelled parts into small diagrams} and Lemma \ref{lemma: generalised Foulkes}, we know that $U_w \tensorover{k\Sigma_n} S_\nu \simeq k\Sigma_l \tensorover{k(\Sigma_m \wr \Sigma_n)} k\Sigma_n \tensorover{k\Sigma_n} S_\nu \simeq H^{(m)}_\nu$. Hence $H^{(m)}_\nu \in \mathcal{F}_{mn}(S)$.
\end{proof}

Lemma \ref{lemma: rcm in F implies gFm in F} provides an idea how to create counterexamples for the general statement.

\begin{thm neg} Let $A= P_k(r,\delta)$ with $\delta \neq 0$. There are $0 \leq n \leq l \leq r$ and partitions $\nu$ of $n$ such that the restriction $e_lAe_n\tensorover{e_nAe_n}S_\nu$ of the cell module $Ae_n \tensorover{e_nAe_n} S_\nu$ does not admit a dual Specht filtration.  
\end{thm neg}

\noindent\emph{Idea of proof.}
By Lemma \ref{lemma: rcm in F implies gFm in F}, finding a generalised Foulkes module with inner twists $H^{(m)}_\nu$ which does not admit a dual Specht filtration if $\text{\textnormal{char}}k = p \geq 5$ creates a counterexample to $e_{mn}Ae_n \tensorover{e_nAe_n} S_\nu \in \mathcal{F}_{mn}(S)$. 

In the literature, it is common to use dual Specht modules to construct cell modules for Brauer or partition algebras, e.g. in \cite{HP}, \cite{HHKP}. This is why we decided to use dual Specht modules as well. However, in the literature about Foulkes modules, it is common to use Specht modules. A proper proof for the general non-existence of dual Specht filtrations would need a counterexample for our generalised Foulkes module with inner twists. We give a counterexample using the usual Specht and generalised Foulkes modules and expect that a similar counterexample exists for the dual case. 

 Let $\text{\textnormal{char}}k = 5$. The generalised Foulkes module (without inner twists) $H_{(1^5)}^{(2)}=k\Sigma_{10} \tensorover{k(\Sigma_2 \wr \Sigma_5)} (S^{(2)} \oslash S^{(1^5)} )$ has a non-projective indecomposable summand $M$ of the form \begin{center}
$ D^{(5,4,1)} \oplus D^{(7,1^3)}$ \\ $D^{(10)} \oplus D^{(8,1^2)} \oplus D^{(6,1^4)} \oplus D^{(5,3,1^2)} \oplus D^{(4,2)}$ \\ $D^{(5,4,1)} \oplus D^{(7,1^3)} $ \end{center}
 as shown by deBoeck in her PhD thesis \cite[Theorem 8.4.8]{deBoeck}, where $D^\lambda$ denotes the simple module corresponding to the partition $\lambda$ and the upper row is the head of $M$, the bottom row its socle. 

Assume that $M$ has a Specht filtration. Then every composition factor must appear in one of the Specht modules of the filtration and the Specht modules appearing in the filtration may not have any other composition factor. $M$ has two simple quotients, namely $ D^{(5,4,1)}$ and $D^{(7,1^3)}$, which must be simple quotients of Specht modules in the filtration as well.

Compute the Specht modules for $\Sigma_{10}$ (e.g. using the hecke package in GAP) and cross out those with a composition factor which is not a composition factor of $M$. The remaining Specht modules are $S^{(10)}$, $S^{(7,1^3)}$, $S^{(6,1^4)}$, $S^{(5,3,1^2)}$ and $S^{(4^2,2)}$. But none of them has simple quotient $D^{(5,4,1)}$. This shows that $M$ and thus $H^{(1^5)}_{(2)}$ does not admit a Specht filtration.

\subsection{On the necessity of the assumptions in Theorem \ref{main theorem}}\label{subsec: necessity of conditions}
Let $\text{\textnormal{char}}k \geq 5$ or $\text{\textnormal{char}}k =0$. 
In Subsection \ref{subsec: proof negative}, we have seen that it is necessary that the generalised Foulkes module with inner twists $H^{(m)}_\nu$ with $mn=l$ admits a dual Specht filtration in order to have that $e_lAe_n \tensorover{e_nAe_n} S_\nu$ admits a dual Specht filtration. This shows that we cannot remove assumption (\ref{main thm assumption gFm}) from Theorem \ref{main theorem} altogether. 
A similar statement concerning usual Foulkes modules explains the assumption (\ref{main thm assumption char}) from Theorem \ref{main theorem}.


If $n=0$ and $l=am$ then $H^{(a^m)}$ is a direct summand of $e_lAe_0$, so in this case we need $H^{(a^m)} \in \mathcal{F}_l(S)$. 
If $\text{\textnormal{char}}k > m$, we know for sure that $H^{(a^m)} \in \mathcal{F}_l(S)$ by Corollary \ref{cor: Foulkes Specht}, but this assumption might be too strong. If $a>2$, the maximal $m$ with $am = l$ is for $a=3$, so $m=\frac{l}{3}$. This shows that assumption (\ref{main thm assumption char}) in Theorem \ref{main theorem} is necessary for $n=0$. 

All other (generalised) Foulkes modules appear as proper exterior tensor factors. It is possible that there are summands $U_v \tensorover{k\Sigma_n} S_\nu$ which are dual Specht filtered even if their tensor factors are not, since the Littlewood-Richardson rule only works in one direction. In general, it may be possible to relax assumption (\ref{main thm assumption char}), but not to drop it completely. Assumption (\ref{main thm assumption gFm}) also cannot be dropped completely, but further insight into generalised Foulkes modules (with inner twists) may allow to verify it under relatively weak assumptions in the future.  \\

\section*{Acknowledgements}
I am grateful to Steffen Koenig for his time, useful discussions and proof reading of this article, to Mark Wildon for giving me his notes on Specht filtrations of generalised Foulkes modules, to Eugenio Giannelli for discussions on Foulkes modules and to Rowena Paget for discussions on generalised Foulkes modules. Furthermore, I would like to thank the reviewer for carefully checking my submissions and providing helpful comments to improve the article.  This article is mostly based on my PhD thesis, which was partly funded by DFG-SPP 1489.  

\bibliographystyle{halpha}
\bibliography{referencesPartitionAlgebra}
\end{document}